\documentclass[a4paper]{amsart}
\pdfoutput=1    

\usepackage[english]{babel}
\usepackage[utf8]{inputenc}
\usepackage{amsmath, amsthm, amssymb, mathtools, stmaryrd, enumitem, xcolor, bm, calc}

\usepackage{tikz}
\usetikzlibrary{matrix, calc, arrows, positioning, cd, shapes.geometric}

\definecolor{linkred}{rgb}{0.6,0,0}
\usepackage[
	pdftitle={The Blowup Formula for the Instanton Part of Vafa--Witten Invariants on Projective Surfaces},
	pdfauthor={Nikolas Kuhn, Oliver Leigh and Yuuji Tanaka},
	ocgcolorlinks,
	linkcolor=black,
	citecolor=linkred,
	urlcolor=linkred]
{hyperref}

\newtheorem{thm}{Theorem}[section]
\newtheorem{cor}[thm]{Corollary}
\newtheorem{lem}[thm]{Lemma}
\newtheorem{prop}[thm]{Proposition}

\theoremstyle{definition}
\newtheorem{definition}[thm]{Definition}

\newtheorem{rem}[thm]{Remark}

\newtheorem{ex}[thm]{Example}
\newtheorem{notation}[thm]{Notation}

\newtheorem{alg}[thm]{Algorithm}

\numberwithin{equation}{section}

\newcounter{situationCounter}

\newtheoremstyle{situation}{}{}{\addtolength{\leftskip}{0em}\itshape}{}{\bfseries}{}{.5em}{\thmname{#1} {#2}}
\theoremstyle{situation}
\newtheorem{situation}[situationCounter]{Situation}

\newcommand{\CC}{{\mathbb{C}}}
\newcommand{\PP}{{\mathbb{P}}}
\newcommand{\QQ}{{\mathbb{Q}}}
\newcommand{\ZZ}{{\mathbb{Z}}}

\newcommand{\pt}{{\mathrm{pt}}}
\newcommand{\T}{{\mathrm{T}}}
\newcommand{\Tvir}{\T^\vir}
\newcommand{\N}{{\mathrm{N}}}
\newcommand{\td}{{\mathrm{td}}}

\newcommand{\psl}{[\hspace{-0.15em}[}
\newcommand{\psr}{]\hspace{-0.15em}]}

\newcommand{\vir}{{\mathrm{vir}}}

\newcommand{\Z}{{\mathsf{Z}}}
\newcommand{\Zhat}{{\widehat{\mathsf{Z}}}}

\newcommand{\torus}{{\widetilde{T}}}
\newcommand{\Mhat}{{\widehat{M}}}
\newcommand{\Xhat}{{\widehat{X}}}
\newcommand{\chat}{{\widehat{c}}}

\renewcommand{\Mc}{{M(c)}}
\newcommand{\Mhatc}{{\Mhat(\chat\hspace{0.1em})}}
\newcommand{\Mhatmc}[1]{{\Mhat^{#1}(\chat\hspace{0.1em})}}

\newcommand{\PPhat}{{\widehat{\mathbb{P}}}}

\newcommand{\eCH}{{\epsilon}}
\newcommand{\eLB}{{\bm{e}}}
\newcommand{\eT}{{\mathbf{e}}}
\newcommand{\tT}{{\mathbf{t}}}

\newcommand{\relhom}{R\mathcal{H}om_{\pi}}
\newcommand{\shom}{\mathcal{H}om}

\newcommand{\vdim}{\mathrm{vdim}}

\DeclareMathOperator{\Hom}{Hom}
\DeclareMathOperator{\Ext}{Ext}
\DeclareMathOperator{\Spec}{Spec}
\DeclareMathOperator{\ch}{ch}
\DeclareMathOperator{\Eu}{e}

\DeclareMathOperator*{\Res}{Res}

\DeclareMathOperator{\coeffs}{F} 

\DeclarePairedDelimiter\abs{\lvert}{\rvert}

\begin{document}

\title[Blowup Formula for Instanton Vafa--Witten Invariants on Surfaces]{The Blowup Formula for the Instanton Part of Vafa--Witten Invariants on Projective Surfaces}
\author{Nikolas Kuhn}
\author{Oliver Leigh}
\author{Yuuji Tanaka}
\email{kuhnn@maths.ox.ac.uk}
\email{oliverle@math.uio.no}
\email{y-tanaka@math.kyoto-u.ac.jp}

\address{Max Planck Institute for Mathematics and Mathematical Institute, University of Oxford}
\address{Department of Mathematics, Uppsala University and Department of Mathematics, University of Oslo}
\address{Department of Mathematics, Kyoto University}
\date{May 2022}

\begin{abstract}
We prove a blow-up formula for the generating series of virtual $\chi_y$-genera for moduli spaces of sheaves on projective surfaces, which is related to a conjectured formula for topological $\chi_y$-genera of Göttsche. Our formula is a refinement of one by Vafa--Witten relating to S-duality. 

We prove the formula simultaneously in the setting of Gieseker stable sheaves on polarised surfaces and also in the setting of framed sheaves on $\PP^2$. The proof is based on the blow-up algorithm of Nakajima--Yoshioka for framed sheaves on $\PP^2$, which has recently been extend to the setting of  Gieseker $H$-stable sheaves on $H$-polarised surfaces by Kuhn--Tanaka. 
\end{abstract}

\maketitle


\section{Introduction}

Hilbert schemes of points on a projective surface have been extensively studied, and while the picture is not complete, their geometry is well understood. In particular, there are many deep results concerning their enumerative invariants (e.g. \cite{G_Betti},\cite{Ch_On},\cite{EGL},\cite{MOP}). For their higher rank analogues, the moduli spaces of (semi-)stable sheaves on a surface, the situation is more complicated. Results for topological invariants for moduli spaces in ranks $2$ and $3$ on some surfaces have been obtained e.g. in \cite{LQ_Blowup2}, \cite{Yo_Be}, \cite{MM_Int} and \cite{Ko_Euler}, but, in general, one does not expect any reasonable formulas for the topological invariants. 

More recently, there has been an interest in studying \emph{virtual} analogues of the topological invariants, which are defined using the perfect obstruction theory coming from the moduli problem. This interest is partially motivated by the $S$-duality conjecture \cite{VW_Strong} concerning Vafa--Witten invariants and by the papers \cite{TT_VW1}, \cite{TT_VW2}, where Tanaka--Thomas proposed a mathematical definition of these invariants as the sum of two parts -- one of which is given by the signed virtual Euler characteristics of moduli spaces of stable sheaves. The virtual invariants have the benefit of being invariant under deformations of the surface that preserve the polarization, and -- strikingly -- they are in many cases predicted to depend only on finitely many basic invariants of the underlying surface. For example, explicit formulas up to rank $5$ have been conjectured for the virtual Euler characteristics in \cite{GKL_Cosmic} (see also \cite{GK_Refined}, \cite{GK_Virtual}, \cite{GK_Rank2}, \cite{GKW_Verlinde}; and \cite{GK_Sheaves} for an excellent survey).

In this paper, we prove the first completely general result for virtual $\chi_y$-genus and Euler characteristic of the higher rank moduli spaces: Namely, we prove a blowup formula for the virtual $\chi_y$-genera of the moduli spaces, and thereby for virtual Euler characteristics. This is a first and important step towards establishing the existing conjectures, and provides further evidence for them. Similar formulas for the topological invariants had been obtained in \cite{LQ_Blowup2} for the virtual Hodge polynomial in rank $2$, and in \cite[Prop. 3.1]{G_Theta} for virtual Poincaré polynomials of arbitrary rank moduli spaces. Our result also confirms part of a conjecture of Göttsche \cite[Rem. 3.2]{G_Theta} (stated more generally for virtual Hodge polynomials) when the moduli spaces are unobstructed, so that their virtual and topological $\chi_y$-genera agree. 

In the course of our proof, we prove a blowup formula for an equivariant analogue of the $\chi_y$-genus defined on moduli spaces of framed sheaves on $\PP^2$ and on its blowup. These moduli spaces of framed sheaves and their invariants are of independent interest in physics through in Nekrasov's conjecture and related topics (see for example \cite{Nek_SW}, \cite{NY_ICOB1}, \cite{NO_Random}, \cite{BE_Whit}, \cite{GNY}, \cite{GL_Toric} and \cite{GNY_SW}).

For the rest of the paper, we fix a positive integer $r>0$, which will denote the rank of the sheaves being considered. As mentioned above, we will prove blow-up formulas for two different kinds of moduli spaces. These two situations are:

\begin{situation}\label{sit: H-stable}  \normalfont{\textbf{- Gieseker stable sheaves} \cite{Mo_Don,KT_Blowup}\textbf{.}} 
    In this situation, $X$ is a smooth projective (connected) surface with a fixed polarisation $H$ and $p:\Xhat \rightarrow X$ is the blow-up of $X$ at a chosen reduced point $\pt\in X$ with exceptional divisor $C$. Moreover, $L_X$ is a fixed line bundle on $X$ with $c_1:=c_1(L_X)$, such that the intersection number $(L_X,H)$ is coprime to $r$.
    
    We consider classes $c:= r+c_1+\ch_2 \in H^*(X,\QQ)$ for various $\ch_2\in H^4(X,\QQ)$. 
    On the blow-up $\Xhat$ we will write $\chat_1:= c_1 - k[C]$ for various $k\in \ZZ$ and consider classes $\chat:= r+\chat_1+\ch_2$ for various $\ch_2\in H^*(X,\QQ)$. 
    In this situation, we consider the following moduli stacks:
    \begin{enumerate}
        \item 
        $\Mc$, the moduli stack of oriented Gieseker $H$-stable sheaves on $X$ with Chern character $c$. Here an oriented sheaf is a pair $(E,\, \psi:\det E\rightarrow L_X)$ where $E$ is a Gieseker $H$-stable sheaf and $\psi$ is an isomorphism. 
        \item 
        $\Mhat(\chat\hspace{0.1em})$ the moduli stack of oriented $p^*H$ slope-stable sheaves on $\Xhat$ with Chern character $\chat$.\footnote{Here, slope stability for the nef line bundle $p^*H$ is the same as Gieseker stability for the polarization $p^*H-\varepsilon [C]$ for any small enough $\varepsilon\in\QQ_{>0}$ depending on $\chat$ (this follows for example from \cite[Prop. 4.3]{K_Perverse}). Since the choice of $\varepsilon$ may depend on $\ch_2$, working with slope stability is slightly more convenient when considering generating series.} In this case, we define an orientation to be an isomorphism $\det E\xrightarrow{\sim} p^*L_X\otimes\mathcal{O}_{\Xhat}(-kC)$ if $\chat_1 = p^*c_1-k[C]$.
    \end{enumerate}
\end{situation}

\begin{situation}\label{sit: framed main} \normalfont{\textbf{- Framed sheaves on $\PP^2$} \cite{NY_ICOB1,NY_ICOB2}\textbf{.}}
    In this situation, we let $X: = \PP^2$ and consider the blow-up $p:\Xhat = \PPhat^2 \rightarrow X$ at $\pt = [1:0:0]$ with exceptional divisor $C$. We set $L_X=\mathcal{O}_{\PP^2}$ and $c_1=0$.
    We consider classes on $X$ of the form $c:= r+0+\ch_2$ for various $\ch_2\in H^4(\PP^2,\QQ)$. 
    For classes on the blow-up we consider $\chat_1:= - k[C]$ for some $k\in \ZZ$ and $\chat:= r+\chat_1+\ch_2$ for various $\ch_2\in H^4(\PPhat^2,\QQ).$ 
    In this situation, we consider the following moduli spaces:
	\begin{enumerate}
		\item 
		$\Mc$, the moduli space of \textit{framed sheaves on $\PP^2$} with Chern character $c$. That is, pairs $(E,\varphi)$ where $E$ is a coherent sheaf on $X$ which is locally free in a neighbourhood of $\ell_\infty := \{[0:z_1:z_2]\}$ and where $\varphi$ is an isomorphism $E|_{\ell_\infty}\overset{\sim}{\rightarrow} \mathcal{O}_{\ell_{\infty}}^{\oplus r}$. 
		\item $\Mhatc$, the moduli space of \textit{framed sheaves on $\Xhat$} with Chern character $\chat$, where the framing is with respect to the preimage $p^{-1}\ell_{\infty}$ of $\ell_{\infty}$ under $p$, which we denote again by $\ell_{\infty}$.  
	\end{enumerate}
\end{situation}

\begin{rem}
    In Situation \ref{sit: H-stable}, we use moduli stacks of oriented sheaves instead of the usual moduli spaces, since the stacks always have universal sheaves, which simplifies some of the presentation. The moduli stacks are degree $1/r$ gerbes over the coarse moduli spaces and all the invariants that we consider differ only by a factor of $1/r$.  
\end{rem}

We will now address how one can define algebraic invariants in Situations \ref{sit: H-stable} and \ref{sit: framed main}. In Situation \ref{sit: H-stable}, this uses that the moduli stacks are proper and carry a perfect obstruction theory, while in Situation \ref{sit: framed main}, one uses the fact that the moduli spaces carry a natural torus action with isolated torus fixed points. 

In what follows below, we let $M$ denote either $M(c)$ or $\Mhatc$ in either situation, we let $\mathcal{E}$ denote the universal sheaf on $M\times X$ or $M\times \Xhat$ and we let $\pi$ denote the projection morphism $\pi:M\times X\to M$ or $\pi:M\times \Xhat\to M$. For $M=M(c)$, take $\chat_1:= c_1$.
\begin{itemize}
    \item[(a)] \emph{Situation A:} $M$ is a proper Deligne--Mumford stack with a canonical perfect obstruction theory of virtual dimension 
		$
		-2r\ch_2  + \chat_1^2 -(r^2-1)\chi(X,\mathcal{O}_X).
		$
		The $K$-theoretic class of the virtual tangent bundle is 
		\[ 
		\T_M^\vir = \mathcal{O}_{M}^{\oplus\chi(X,\mathcal{O}_X)}- R\pi_*R\shom(\mathcal{E},\mathcal{E}).
		\]
	\item[(b)] \emph{Situation B:} $M$ is representable by a smooth quasi-projective variety with dimension 
		$
		-2r\ch_2  + \chat_1^2. 
		$
		There is a canonical action of $\torus := (\CC^*)^{2+r}$ on $M$ whose fixed locus is a set of isolated points. This action extends in a canonical way to $M\times X$ (respecctively $ M\times \Xhat$), such that $\mathcal{E}$ has a canonical structure of equivariant sheaf.   
		The $\torus$-equivariant class of the tangent bundle in $K$-theory is 
		\[ 
		\T^{\torus}_M = - R\pi_*R\shom(\mathcal{E},\mathcal{E}(-\ell_{\infty})).
		\]
\end{itemize}

As references for these properties: For Situation \ref{sit: H-stable}, see either \cite[\S 5.6]{Mo_Don} or \cite[Cor. 3.26 \& Prop. 5.12]{KT_Blowup}  and for Situation \ref{sit: framed main} see \cite[\S 2 \& \S 3]{NY_ICOB1} and  \cite[\S 5]{NY_PCSOB2}.
Using these properties, we now define a notion of ``integration'' for Chow cohomology classes on moduli spaces in each situation. We will let $\coeffs$ denote a field of characteristic zero, and write $A^*$ for Chow cohomology rings with coefficients in $\coeffs$.
\begin{notation} \label{not: integral notation}
Let $M$ be a moduli space from either Situation \ref{sit: H-stable} or \ref{sit: framed main}. Define the following integral notation:
\begin{enumerate}
	\item[(a)] \textit{Situation \ref{sit: H-stable}:} Using the perfect obstruction theory, we obtain a virtual fundamental class $[M]^\vir$. Using this, we define for $\alpha \in A^*(M)$:
	\[
	\int_M\alpha := \int_{[M]^\vir} \alpha :=\deg\left(\alpha \cap [M]^{\vir}\right),
	\]
	where for a Chow class $\gamma\in A_*(M)$ we let $\deg\gamma$ be the number obtained by pushforward along the structure morphism $M\to \Spec \CC$.
	\item[(b)] \textit{Situation \ref{sit: framed main}:} In this situation the moduli space is non-proper, but has a proper $\torus$-fixed locus. We define integration via the Atiyah--Bott localisation formula. Hence for a $\torus$-equivariant class $\alpha \in A_{\torus}^*(M)$:
	\[
	\int_M\alpha := \sum_{x\in M^{\torus}} \frac{\alpha|_{x}}{\Eu\left(\T_M^{\torus}|_{x}\right)}.
	\]
\end{enumerate}
Moreover, in light of the unified notations of Situations \ref{sit: H-stable} and \ref{sit: framed main} we suppress  the equivariant notation when considering Situation \ref{sit: framed main} unless it is necessary. Hence, from now on, in Situation \ref{sit: framed main} we write  $A^*(M)$ instead of $A^*_{\torus}(M)$, $\T_M$ instead of $\T_M^{\torus}$ and always consider $\mathcal{E}$ as an equivariant sheaf. 
\end{notation}

\begin{rem} 
	If we consider the moduli spaces in Situation \ref{sit: framed main} as having the trivial perfect obstruction theory, then we have $\T^\vir_M = \T_M$ and it makes sense to use the word \textit{virtual} when discussing the enumerative invariants of both situations. 
\end{rem}

Using the above notation, in this article, we will consider invariants of the form 
\[
    \int_{M} \Theta(\T^{\vir}), 
\]
where $\Theta: K(M)\to A^*(M)$ is some characteristic class. In particular, we are interested in the case where $\Theta = \Theta_y$ is the multiplicative class defined on line bundles by $\Theta_y(L) = (1-y\ch(L)^{-1})\cdot \td(L)$, where $y$ is a formal variable. In this case, we set
\[
    \chi_{-y}^{\vir}(M):= \int_{M} \Theta_y(\T^{\vir})\in \QQ[y].
\]
This is the \emph{virtual $\chi_y$-genus} of $M$ as defined by Fantechi--Göttsche in \cite[\S 4]{FG_RR}. The main result of the article is the following theorem: 

\begin{thm}[Main Theorem] \label{thm: main theorem}
    Consider Situation \ref{sit: H-stable} or \ref{sit: framed main} such that $0\leq k <r$. Consider the generating series of virtual $\chi_y$-genera 
    \begin{align*}
        \Z(q,y) 
        &:= 
        \sum_{\substack{\ch_2}} \chi^\vir_{-y}\big(M(r+c_1+\ch_2)\big) \,q^{\mathrm{vdim}\,M(r+c_1+\ch_2)}, \hspace{1em}\mbox{and}
        \\
        \Zhat(q,y) 
        &:=
        \sum_{\substack{\ch_2 }} \chi^\vir_{-y}\Big(\Mhat(r+\chat_1+\ch_2)\Big) \,q^{\mathrm{vdim}\,\Mhat(\,r+\chat_1+\ch_2\,)}.
	\end{align*}
      Then, we have the identity 
	\[
        \Zhat  = \mathsf{Y}_{k}  \cdot \Z,
	\]
	where 
    \[
	   \mathsf{Y}_k = \mathsf{Y}_{k}(q,y)
        := \prod_{n>0} \big( 1- (q^{2}y)^{rn} \big)^{-r}
        \sum_{k_1+\cdots+k_r = k} (q^2y)^{\sum_{ i<j} (k_i-k_j)^2/2 } \, y^{\sum_{ i<j} (k_i-k_j)/2 }
	    .
	\]
\end{thm}

\begin{rem}
The factor $\mathsf{Y}_{k} $ appearing in Theorem \ref{thm: main theorem} is independent of all input data except $k$ and the rank $r$. This is the same  blow-up behaviour observed for the topological invariants studied in \cite{G_Betti},\cite{Ch_On},\cite{LQ_Blowup1,LQ_Blowup2} and \cite{G_Theta}. 
\end{rem}

\begin{rem}
The result in Theorem \ref{thm: main theorem} was originally conjectured in Situation \ref{sit: H-stable} for the topological $\chi_y$-genus by Göttsche in \cite[Rem. 3.2]{G_Theta} (stated more generally for virtual Hodge polynomials). The factor $\mathsf{Y}_{k}$ can be rewritten in Göttsche's notation as follows. Let $A = (a_{ij} )_{ij}$ be the $(r-1)\times (r-1)$-matrix with entries $a_{ij} = 1$ for $i\leq j$ and $a_{ij} = 0$ otherwise, and let $I$ be the column vector of length $r-1$ with all entries equal to one. Let $\eta$ denote the Dedekind eta function. Then we have
\[
    \mathsf{Y}_{k}
    =
    \frac{(q^{2r}y^r)^{r/24}}{\eta(q^{2r}y^r)^r} 
    \sum_{v\in \ZZ^{r-1} + \frac{k}{r}I} (q^{2r}y^r)^{v^t A v} \,y^{v^t A I}. 
\]
In the cases where the moduli problem is unobstructed (i.e. if every stable sheaf $E$ on $X$ and $\Xhat$ has $\Ext^2(E,E)=0$) topological and virtual $\chi_y$-genus coincide. Hence, in these cases, Theorem \ref{thm: main theorem} proves Göttsche's topological conjecture. 
\end{rem}

One can also consider virtual Euler characteristic and virtual holomorphic Euler characteristic. These are defined respectively as
\[
e^\vir(M) 
:=
\int_M \Eu(\T^\vir_M) 
\hspace{1em}\mbox{and}\hspace{1em}
\chi^\vir(M) 
:=
\int_M \td(\T^\vir_M) 
\]
which are the respective specialisations $\chi_{-1}^\vir$ and $\chi_0^\vir$ of $\chi_y^{\vir}$. Using these characterisations, Theorem \ref{thm: main theorem} has the following corollary. 

\begin{cor} \label{cor:main cor}
	In the setting of Theorem \ref{thm: main theorem}, using appropriately defined generating series, we have
    \[
        \Zhat_{e^\vir}  = \mathsf{Y}_{k,\,e^\vir}  \cdot \Z_{e^\vir} 
        \hspace{1cm}\mbox{and}\hspace{1cm}
        \Zhat_{\chi^\vir}  = \mathsf{Y}_{k,\,\chi^\vir}  \cdot \Z_{\chi^\vir},
	\]
	where 
	\[
        \mathsf{Y}_{k,\,e^\vir} = 
        \prod_{n>0} \big( 1- q^{2rn} \big)^{-r}
	    \sum_{\sum_{i=1}^r k_i = k} q^{\sum_{ i<j} (k_i-k_j)^2 } 
        \hspace{1.45em}\mbox{and}\hspace{1.45em}
        \mathsf{Y}_{k,\,\chi^\vir} =\begin{cases} 1, & \mbox{if } k=0;\\ 0, &\mbox{else.} 
        \end{cases}
	\]
\end{cor}

\begin{rem}
    The case $k=0$ of the blow-up formula for virtual holomorphic Euler characteristics of Corollary \ref{cor:main cor} was previously proved by Nakajima--Yoshioka in \cite[Thm. 2.11, i)]{NY_PCSOB3} for Situation \ref{sit: framed main}. 
\end{rem}

\subsection*{Outline of the Article}
In order to prove Theorem \ref{thm: main theorem}, we use the blow-up algorithm of Nakajima--Yoshioka \cite{NY_PCSOB3} (extended to Situation \ref{sit: H-stable} by \cite{KT_Blowup}) to obtain a weak universal blowup formula for invariants of the form being studied. A key aspect of this formula is its universality, namely it shows that the desired blow-up formula exists and is independent of much of the input data. In particular, the formula will be the same in Situation \ref{sit: H-stable} and Situation \ref{sit: framed main}. We then analyse the equivariant geometry arising in Situation \ref{sit: framed main} to determine the exact form of the desired formula. This proceeds via a reduction to rank $1$, where we can use the $(q,t)$-Nekrasov--Okounkov formula \cite[Thm. 1.3]{Ra_Wa} to conclude.

The article is structured as follows:
\begin{itemize}
    \item \textit{Section \ref{sec: NY blow-up}}: We recall the blow-up algorithm of Nakajima--Yoshioka \cite{NY_PCSOB3}  and show how it is extended to both situations by \cite{KT_Blowup}. 
    \item \textit{Section \ref{sec: weak structure thm}}: We state and  prove a weak universal blowup formula for invariants defined by classes which are multiplicative in the tangent bundle. 
    \item \textit{Section \ref{sec: equivariant proof}}: We prove the \hyperref[thm: main theorem]{Main Theorem} by determining the coefficients from the weak universal blowup formula. We do this by examining the equivariant structure arising in Situation \ref{sit: framed main}. 
\end{itemize}

\hypertarget{sec: basic notation}{}

\subsection*{Basic Notation}
We fix a positive integer $r>0$ and will always assume the notation described in Situations \ref{sit: H-stable} and \ref{sit: framed main}. In particular, we consider classes of the form $c:= r+c_1+\ch_2 \in A^*(X)$ and $\chat:= r+\chat_1+\ch_2 \in A^*(\Xhat)$ described there. We will also use the integral notation described in Notation \ref{not: integral notation}. 

Other basic notation includes:
\begin{itemize}
    \item $C_{m}:=\mathcal{O}_C(-m-1):=\mathcal{O}_C\otimes \mathcal{O}_{\Xhat}((m+1)\,C)$ denotes the unique degree $-m-1$ line bundle on the exceptional divisor $C$. In Situation \ref{sit: framed main}, we endow $C_m$ with the equivariant structure obtained from the inclusion into $\mathcal{O}_C$ if $m\geq -1$ (resp. the dual one for $m<-1$).
    \item $\eCH_m:=\ch(j_*C_m) = [C]-(m+1/2)[\pt]$ is the Chern character of $C_m$ taken as a coherent sheaf on $\Xhat$  via the inclusion $j:C\hookrightarrow \Xhat.$
    \item For a moduli space $M$ of sheaves on $X$ (or $\Xhat$) denote the universal sheaf on $M\times X$ (or $M\times \Xhat$) by $\mathcal{E}$ and let $\pi:M\times X\rightarrow M$ (resp. $\pi:M\times \Xhat\to M$) be the projection. We will denote $\Tvir_M$ simply by $\Tvir$.
    \item $\relhom(\mathcal{F},\mathcal{G}) := R\pi_* R\shom(\mathcal{F},\mathcal{G})$. 
    \item For a $\CC^*$-representation $\hbar$ we write $\eLB^\hbar$ for the associated equivariant line bundle.
    \item $\coeffs$ denotes the coefficient field for Chow cohomology groups.
    \item In Situation \ref{sit: framed main}, we take $\torus \coloneq (\CC^*)^{2+r}$. The $\torus$-action on $M$ is explained in Definition \ref{def: torus action}.
    \item $\tT_1,\tT_2, \eT_1,\ldots, \eT_r$ are the fundamental one dimensional $\torus$-representations. See Notation \ref{not: T representations} for more details.
    \item For a coherent sheaf $\mathcal{F}$ define the notation in Situation \ref{sit: H-stable}: 
    \[
        \Tvir_{\mathcal{F}}:= - \mathcal{O}_{M}^{\oplus\chi(X,\mathcal{O}_X)}- R\pi_*R\shom(\mathcal{F},\mathcal{F}),
    \]
    and in Situation \ref{sit: framed main}:
    \[
    \Tvir_{\mathcal{F}}:= - R\pi_*R\shom(\mathcal{F},\mathcal{F}(-\ell_{\infty})).
    \]
\end{itemize}

\subsection*{Acknowledgements}
The authors would like to thank Hiraku Nakajima and Martin Raum for useful conversations which contributed to this article. 

This project was started while the first and second-named authors were Junior Fellows in the program \textit{Moduli and Algebraic Cycles}; taking place at Institut Mittag-Leffler between the 30th of August 2021 and the 10th of December 2021. They would like to thank the organizers John Christian Ottem, Dan Petersen, David Rydh for creating a wonderful environment during the workshop and for helpful mathematical conversations. The third-named author was partially supported by the JSPS Grant-in-Aid for Scientific Research numbers JP16K05125, 16H06337, 20K03582, 20K03609, 21H04994 and 21K03246. He is grateful to Hiraku Nakajima for inviting him to Kavli IPMU in Autumn 2020 and invaluable discussions during the visit.

\section{Nakajima--Yoshioka blow-up algorithm} \label{sec: NY blow-up}

In \cite{NY_PCSOB3}, Nakajima and Yoshioka introduced an algorithm to express certain tautological integrals on $\Mhatc$ in terms of similar integrals on $M(c+n[\pt])$ for various $n$, which applies in Situation \ref{sit: framed main}. The results necessary to apply the algorithm in Situation \ref{sit: H-stable} were established in \cite{KT_Blowup}. In this section we will review the Nakajima--Yoshioka blow-up algorithm in both situations simultaneously and draw some conclusions. First, we recall the concept of $m$-stable sheaves.

\begin{definition}[\textit{$m$-stable sheaves}]
    \cite[Lem. 3.29]{KT_Blowup} \& \cite[\S 1.1]{NY_PCSOB3}
    Let $m\in \ZZ$. We denote by $\Mhatmc{m}$ the moduli stack of $m$-stable sheaves which parameterises: 
	\begin{enumerate}
    	\item 
    	\textit{In Situation \ref{sit: H-stable}:}  Sheaves $E$ on $\Xhat$ with Chern character $\chat$ such that
		the torsion-free quotient of the coherent sheaf $p_*E$ is $H$-stable on $X$, and such that $\Hom(\mathcal{O}_C(-m),E)=0$ and $\Hom(E,\mathcal{O}_C(-m-1))=0$.
		\item 
		\textit{In Situation \ref{sit: framed main}:}
		Framed sheaves $(E,\Phi)$ on $\PPhat^2$ such that $\ch(E)=\chat$ and $E$ is
		torsion-free away from the exceptional divisor $C$, and such that $\Hom(\mathcal{O}_C(-m),E)=0$ and $\Hom(E,\mathcal{O}_C(-m-1))=0$. 
	\end{enumerate}
\end{definition}

\begin{rem}\label{rem: extending integral to m-stable}
    The discussion preceding Notation \ref{not: integral notation} also holds for $M = \Mhatmc{m}$: In Situation \ref{sit: H-stable}, they are proper Deligne--Mumford stacks with a canonical perfect obstruction theory. In Situation \ref{sit: framed main}, they are smooth, quasi-projective varieties and carry a canonical $\torus$-action with isolated fixed points. In either case, the (virtual) dimension is the same as the one of $\Mhatc$.  Hence, we can extend the integral notation from Notation \ref{not: integral notation}  to  $M = \Mhatmc{m}$. 
\end{rem}

In order to simplify the exposition, we will only state the results of this section for spaces of $m$-stable sheaves. This is justified by the following result. 

\begin{prop} \label{prop: stable as m-stable} 
    Suppose we are in either Situation \ref{sit: H-stable} or \ref{sit: framed main}.
	\begin{enumerate}[label=(\roman*)]
		\item \cite[Prop. 3.25 \& Cor. 3.27]{KT_Blowup} \& \cite[Prop. 3.3]{NY_PCSOB2}\label{mstabi}
		If $\chat$ is of the form $p^*c$ then there is a natural isomorphism $p_*(-):\Mhatmc{0} \overset{\sim}{\rightarrow} \Mc$ whose inverse is $p^*(-):\Mc \overset{\sim}{\rightarrow}  \Mhatmc{0}$.
		\item \cite[Prop. 3.37]{NY_PCSOB2} \& \cite[Prop. 7.1]{NY_PCSOB1}\label{mstabii} For each $c$ there is an $m'\geq  0$ such that $\Mhatmc{m}=\Mhatc$ for all $m\geq m'$. 
		\item In Situation \ref{sit: H-stable}, the isomorphisms in \ref{mstabi} and \ref{mstabii} preserve the perfect obstruction theories. In Situation \ref{sit: framed main}, the isomorphisms are compatible with the $\torus$-equivariant structure.  
	\end{enumerate}
\end{prop}

The Nakajima--Yoshioka blow-up algorithm will be applied to integrals with the following type of integrand classes. 

\begin{notation} \label{not: phi notation}
    Suppose we are in either Situation \ref{sit: H-stable} or \ref{sit: framed main} and let $M$ be $\Mhatmc{m}$ with $\pi:\Xhat \times M \rightarrow M$ as the projection.
    We use the symbol $\Phi$ to denote a rule that associates a class $\Phi(\mathcal{F})\in A^*(M)$ to each coherent sheaf $\mathcal{F}$ on $\Xhat\times M$. 
    More precisely, we require that $\Phi(\mathcal{F})$ is given as a power series in classes of the form $\pi_{*}(c_{i_1}(\mathcal{F})\cdots c_{i_N}(\mathcal{F})\cdot\alpha)$ for some list of integers $(i_1,\ldots,i_N)$ and some Chow class $\alpha\in A^*(\Xhat)$.
    
    We also use this notation in the more general setting where $M'$ is a smooth projective scheme over $M$ that either carries an obstruction theory compatible with the one on $M$ (in Situation \ref{sit: H-stable}) or a $\torus$-action for which the map $M' \to M$ is equivariant (in Situation \ref{sit: framed main}).
\end{notation}

 \begin{ex} \label{ex: RHom}
    In both Situation \ref{sit: H-stable} and \ref{sit: framed main}, any polynomial in Chern classes of  $\Tvir$,
    $\relhom(\mathcal{E},C_m)$ or 
    $\relhom(C_m,\mathcal{E})$ will satisfy the conditions of Notation \ref{not: phi notation} for $\Phi(\mathcal{E})$. This follows from the virtual and equivariant versions of the Grothendieck-Riemann-Roch Theorem, and the characterisation of the respective tangent bundles given by Remark \ref{rem: extending integral to m-stable} and the discussion preceding Notation \ref{not: integral notation}.
 \end{ex}

In what follows, it will be useful to better understand the $K$-theory classes mentioned in Example \ref{ex: RHom}. 

\begin{lem} \label{lem: RHom properties}
    Suppose we are in Situation \ref{sit: H-stable} or \ref{sit: framed main} and consider the moduli space $M=\Mhatmc{m_0}$ with universal sheaf $\mathcal{E}$ and projection $\pi:\Xhat\times M\to M$. Then we have:
    \begin{enumerate}[label = (\roman*)]
        \item\label{RHomprop1} For any $m\in \ZZ$, the respective $K$-theoretic ranks of $\relhom(\mathcal{E},C_m)[1]$ and 
        $\relhom(C_m,\mathcal{E})[1]$ are given by
        \[
            rm+(\chat_1, [C]) 
            \hspace{1em}\mbox{and}\hspace{1em}
             r(m+1)+(\chat_1, [C] ).
        \]
        \item \label{RHomprop2} For $m=m_0$, we have the vector bundles:\vspace{0.3em}
        
        $\relhom(\mathcal{E},C_m)[1] = \Ext^1_{\pi}(\mathcal{E},C_m)$ and \vspace{0.3em}
        
        $\relhom(C_{m-1},\mathcal{E})[1]=\Ext^1_{\pi}(C_{m-1},\mathcal{E}).$\vspace{0.3em}
        \item \label{RHomprop3}
        Suppose that $\chat$ satisfies $(\chat_1,[C]) = 0$, and consider the space $\Mhatmc{0}$. Then for $m\in \ZZ$  we have the following  identities of $K$-theory classes:
    	\begin{align*}
    	    &\relhom(C_{m},\mathcal{E}) 
    	    =
    	    -\tT_1\tT_2R\Gamma(C_{m+1})^{\vee}\otimes\,     \relhom(C_{0},\mathcal{E}), \mbox{ and }
    	    \\
    	    &\relhom(\mathcal{E}, C_{m}) 
    	    =
    		-R\Gamma(C_m)\, \otimes \, \relhom(C_{0},\mathcal{E})^\vee \,.
    	\end{align*}
       In Situation \ref{sit: framed main} this uses the equivariant structure defined in the \hyperlink{sec: basic notation}{Basic} \hyperlink{sec: basic notation}{Notation}, and the $R\Gamma(C_m)$ are $\torus$-representations explicitly given by \eqref{eq: pushforward formula}.
    	In Situation \ref{sit: H-stable}, the identity holds if one replaces both $\tT_1$ and $\tT_2$ with the trivial line bundle.
    	
    	\item \label{RHomprop4} Let $m_1, m_2$ be integers. The object $R\Hom_{\Xhat}(C_{m_1},C_{m_2})$ has $K$-theoretic rank equal to $-1$. In Situation \ref{sit: framed main}, the $K$-theory class of $R\Hom_{\Xhat}(C_{m_1},C_{m_2})$ equals a sum of terms of the form $\pm \tT_1^i\tT_2^j$. 
    \end{enumerate}
\end{lem}
\begin{proof}
    Point \ref{RHomprop1} follows from a Grothendieck--Riemann--Roch computation. Point \ref{RHomprop2} follows from the definition of $m$-stability as explained prior to Lemma 3.22 in \cite[p. 79]{NY_PCSOB2}. For point \ref{RHomprop3}, using Bott's formula for pushforwards in equivariant $K$-theory, we have for any $E\in K^{T^{2}}(\PPhat^2)$:
    \[R\Gamma(E) = \sum_F \frac{\iota_F^*(E)}{\Lambda_{-1} \T^*_{F}\PPhat^2}, \]
    where the sum ranges over the torus fixed points of $\PPhat^2$.
    Using this, one can compute
    \begin{equation}\label{eq: pushforward formula}
    \begin{aligned}R\Gamma(\mathcal{O}_C(m))& = \frac{\tT_1^{-m}-\tT_1^{-m-1}}{1-\tT_1^{-1}-\tT_1\tT_2^{-1}+\tT_2^{-1}}+\frac{\tT_2^{-m}-\tT_2^{-m-1}}{1-\tT_2^{-1}-\tT_2\tT_1^{-1}+\tT_1^{-1}}\\
    &= \begin{dcases}
    -\tT_1\tT_2 \, \left(\sum_{i=0}^{-m-2}\tT_1^{i}\tT_2^{-m-2-i}\right),  & \mbox{ if } m<0;\mbox{ and} \\
    (\tT_1\tT_2)^{-m}\left(\sum_{i=0}^m\tT_1^i\tT_2^{m-i}\right), & \mbox{ if } m\geq 0.
    \end{dcases}
    \end{aligned}
    \end{equation}
    By the equivariant localization formula in $K$-theory, we also have  $Rp_*(\mathcal{O}_C(m)) = R\Gamma(\mathcal{O}_C(m)) \mathcal{O}_{\pt}$. These facts also hold non-equivariantly on any surface if one imposes $\tT_1=\tT_2=\CC$. Now, \ref{RHomprop3} follows from a calculation using equivariant Serre duality and the observation that, on $\Mhat^0(p^*c)$, we have the adjunction formula $\relhom(p^*\mathcal{F},C_m) = \relhom(\mathcal{F},p_*C_m)$, where $\mathcal{F}$ is the universal sheaf over $\PP^2\times M_{\PP^2}(c)$. Point \ref{RHomprop4} follows similarly, using that $K$-theoretically, we have $C_m = \mathcal{O}_{\Xhat}((m+1)\,C)-\mathcal{O}_{\Xhat}(mC)$ .
\end{proof}

The Nakajima--Yoshioka blow-up algorithm allows us to express intersections of the form  $\int_{\Mhatmc{m}}\Phi(\mathcal{E})$  in terms of integrals over $M(c')$ for various $c'$. Indeed, it gives a consistent way to compute classes $\Phi_n(\mathcal{E})$ such that
\[
\int_{\Mhatmc{m}}\Phi(\mathcal{E}) 
= 
\sum_{n\geq 0} \int_{M(c+n\pt)}\Phi_n(\mathcal{E}).
\]
For special choices of $\Phi$, we can say much more about the classes $\Phi_n$ (see Theorem \ref{thm: omega_j structure theorem}). We now state the key components of the algorithm.

\begin{thm}[Wall-Crossing Formula] \cite[Thm. 1.1]{KT_Blowup} \& \cite[Thm. 1.5]{NY_PCSOB3} \label{thm: wall crossing}
    Suppose we are in Situation \ref{sit: H-stable} or \ref{sit: framed main}, and let $\Phi(\mathcal{E})$ be as in Notation \ref{not: phi notation}. Then we have the following identity:
    \begin{align*}
        &\int_{\Mhatmc{m+1}}\Phi(\mathcal{E}) - \int_{\Mhatmc{m}}\Phi(\mathcal{E}) 
        \\
        &=
        \sum_{j > 0}\frac{1}{j!} \int_{\Mhat^m(\chat-j\eCH_m)} \Res\limits_{h_{j}=0} \cdots \Res\limits_{h_1=0}
        \,\Phi\Big(\mathcal{E}\oplus \bigoplus_{i=1}^j C_m\otimes \eLB^{-h_i}\Big)
        \,\Psi_m^j(\mathcal{E}),
    \end{align*}
    where $C_m := \mathcal{O}_C(-m-1)$, $\eCH_m := \ch(C_m)$, and where
    \[
        \Psi_m^j(\mathcal{E})
        :=
        \dfrac{\prod_{1\leq i_1\neq i_2\leq j}(h_{i_1}-h_{i_2})}{\prod_{i=1}^j \Eu(-\relhom(\mathcal{E},C_m)\, \eLB^{-h_i}) \Eu(-\relhom(C_m,\mathcal{E})\,\eLB^{h_i}) }.
    \]
\end{thm}

\begin{rem} \label{rem: wall crossing remarks}
    In regards to Theorem \ref{thm: wall crossing}:
    \begin{enumerate}
        \item The notation for taking the residue at $h_i=0$  means: \textit{Taking the coefficient of $h_i^{-1}$ after expanding at $h_i=0$.}
        \item
            We have written the right hand side of Theorem \ref{thm: wall crossing} as an infinite sum. However, from Remark \ref{rem: extending integral to m-stable} and the discussion preceding Notation \ref{not: integral notation} we know that the (virtual) dimension of $\Mhat^m(\chat-j\eCH_m)$ decreases strictly as $j\gg 0$ increases. Hence, for any choice of $\chat$, the right hand side is a finite sum. 
    \end{enumerate}
\end{rem}

\begin{rem}
In the version of Theorem \ref{thm: wall crossing} stated in \cite{NY_PCSOB3}, only the action of a certain sub-torus of $\widetilde{T}=(\CC^{*})^{r+2}$ is considered. In correspondence with H. Nakajima, it was explained to us that this restriction is inessential, and that the results still hold when one considers the whole $\widetilde{T}$-action. 
\end{rem}

\begin{lem}[Twisting by $\mathcal{O}(C)$] \label{lem: Twisting by O(C)}
    \cite[Rem. 3.31]{KT_Blowup} \& \cite[\S 1.5.4.]{NY_PCSOB3}
    Suppose we are in Situation \ref{sit: H-stable} or \ref{sit: framed main}. For each $n\in \ZZ$, there is a natural isomorphism
    \[
        \Mhatmc{m}\xrightarrow{\sim} \Mhat^{m+n}\big(\chat\cdot \ch(\mathcal{O}_{\Xhat}(nC)\big)
    \]
    defined by sending a family of $m$-stable sheaves $\mathcal{F}$ to its twist  $\mathcal{F}(n C)$. In Situation \ref{sit: H-stable}, this isomorphism preserves the perfect obstruction theory and in Situation \ref{sit: framed main}, it is compatible with the $\torus$-equivariant structure.  
\end{lem}

\begin{rem}
    Using the notation of Lemma \ref{lem: Twisting by O(C)}, suppose that $\widetilde{c} = \chat\cdot \ch(\mathcal{O}_{\Xhat}(nC))$. Then we have $(\widetilde{c}_1,[C]) = (\chat_1, [C]) - rn$. 
\end{rem}

The final ingredient is the following.

\begin{prop}[Grassmann-Bundle Formula] \label{prop: Gr bundle formula} 
    Suppose that we are in Situation \ref{sit: H-stable} or \ref{sit: framed main}, and that $\chat$ satisfies $(\chat_1,[C]) = -k$ for some $0< k< r$. Let $\Phi(\mathcal{E})$ be as in Notation \ref{not: phi notation}. Then 
    \[
        \int_{\Mhatmc{1}} \Phi(\mathcal{E}) 
        = 
        \frac{1}{{k}!}
        \int_{\Mhat^1(\chat - {k} \eCH_0)}
        \,\Res_{h_{{k}}=0}\cdots \Res_{h_1=0}
        \Phi\Big(\mathcal{E}\oplus \bigoplus_{i=1}^{{k}} C_0 \otimes \eLB^{-h_i}\Big) \,\overline{\Psi}^{{k}}(\mathcal{E}),
    \]
    where $C_0 = \mathcal{O}_C(-1)$ and where 
    \[
        \overline{\Psi}^{{k}}(\mathcal{E}) := \dfrac{\prod_{1\leq i_1\neq i_2\leq {k}}(h_{i_1}-h_{i_2})}{\prod_{i=1}^{k}  \Eu(-\relhom(C_{0},\mathcal{E})\,\eLB^{h_i})}.
    \]
\end{prop}
\begin{proof}
    This is a reformulation of the results from \cite[Thm. 1.2 \& Prop. 3.36]{KT_Blowup} and \cite[Prop. 1.2]{NY_PCSOB3}. More precisely, by the results in \cite{KT_Blowup} and \cite{NY_PCSOB3}, we have 
    \[\int_{\Mhatmc{1}} \Phi(\mathcal{E}) 
        = \int_{\Mhat^1(\chat - {k} \eCH_0)} q_*\Phi(q^*\mathcal{E}\oplus  C_0\otimes\mathcal{V}), \]
    where $q:\operatorname{Gr}({k}, \Ext^1_{\pi}(C_{0},\mathcal{E}) ) \to \Mhat^1(\chat - {k} \eCH_0)$ is the relative Grassmannian of ${k}$-planes and $\mathcal{V}$ is the associated universal sub-bundle. 
    
    Let $r$ denote the rank of $\Ext^1_{\pi}(C_0,\mathcal{E})$. 
    By the formula of \cite[Example 1]{DP_Flag}, $q_*\Phi(q^*\mathcal{E}\oplus  C_0 \otimes\mathcal{V})$ is equal to the following expression:
    \[
        \operatorname{Coeff}_{h_1^{r-1}\cdots h_k^{r-k}}\hspace{-4pt}\left[{\Phi\Big(q^*\mathcal{E} \oplus \bigoplus_{i=1}^k C_0\otimes \eLB^{-h_i}\Big) \hspace{-5pt}\prod_{1\leq i_1 < i_2 \leq k} (h_{i_1}-h_{i_2}) \hspace{-3pt}\prod_{1\leq i\leq k}} {s_{\frac{1}{h_i}}(\Ext^1_{\pi}(C_0,\mathcal{E}))}\right]
    \]  
    Here, $s_{1/h_i}$ denotes the Segre polynomial in the variable $1/h_i$. We have
    \[
        s_{\frac{1}{h_i}}\big(\Ext^1_{\pi}(C_0,\mathcal{E})\,\eLB^{h_i}\big) =h_i^{r} \Eu\big(\Ext^1_{\pi}(C_0,\mathcal{E})\,\eLB^{h_i}\big)^{-1},
    \]
    and $\Ext^1_{\pi}(C_0,\mathcal{E}) = -\relhom(C_{0},\mathcal{E})$,
    from which we find that $q_*\Phi(q^*\mathcal{E}\oplus  C_0 \otimes\mathcal{V})$ is 
    \begin{align}
        \operatorname{Coeff}_{h_1^{-1}\cdots h_k^{-k}}\hspace{-4pt}\left[\frac{\Phi\Big(q^*\mathcal{E} \oplus \bigoplus_{i=1}^k C_0\otimes \eLB^{-h_i}\Big) \prod_{1\leq i_1 < i_2 \leq k} (h_{i_1}-h_{i_2}) } {\Eu(-\relhom(C_{0},\mathcal{E})\,\eLB^{h_i})}\right].
        \label{eq: Grassmannian formula as coeff}
    \end{align}

    We claim that averaging the expression in \eqref{eq: Grassmannian formula as coeff} over all permutations of $1,\ldots,k$ recovers the iterated residue term in the statement of the proposition. Indeed, in our case the expansion at $h_1 = \ldots = h_k = 0$ is independent of the ordering of the $h_i$, so the iterated residue is simply the coefficient of $h_1^{-1}\cdots h_k^{-1}$. Now the claim  follows from the Vandermonde identity
    \[\sum_{\sigma\in \Sigma_k}\operatorname{sgn} \sigma\; h_{\sigma 1}^0\ldots h_{\sigma k}^{k-1} = \prod_{1\leq i_1<i_2\leq k}(h_{i_2}-h_{i_1}).\]
\end{proof}

\begin{alg}[Blow-Up Algorithm For Single Terms] \label{alg: NY blow-up single term}
    Let $0\leq \ell < r$ and $n\geq 0$ be integers
    and set 
    $\chat :=p^*c-\ell \eCH_0 + n[\pt]$.
    
    \begin{itemize}[leftmargin=2em]
        \item[$\blacktriangleright$] 
        \noindent{\ttfamily INPUT:} The input  for the algorithm is an integral of the form
        \[
            \int_{\Mhatmc{m}} \Phi(\mathcal{E}).
        \]
        \item[$\blacktriangleright$] 
        \noindent{\ttfamily OUTPUT:} The output is an identity of the form
        \[
                \int_{\Mhatmc{m}}\Phi(\mathcal{E})
                = 
                \int_{M(c+n[\pt])}\Phi'(\mathcal{E}) 
                + 
                \sum_{i\in I} 
                \int_{M_i}\Phi_i(\mathcal{E}),
        \]
        where
        \begin{enumerate}[label=(\roman*)]
            \item
            $\Phi'$ is as in Notation \ref{not: phi notation}, and is zero if $0<\ell<r$. 
            \item
            $I$ is a finite collection indexing pairs $(M_i,\Phi_i)$ where $\Phi_i$ is as in Notation \ref{not: phi notation} and $M_i$ is of the form $\Mhat^{m_i}(b_i)$ with $b_i=p^*c-\ell_i \eCH_0 + n_i[\pt]$ for integers $0\leq \ell_i < r$ and $n_i\in \ZZ$ such that 
                $
                    \vdim\,\Mhat^{m_i}(b_i) 
                    ~<~
                    \vdim\,\Mhatmc{m}.
                $
        \end{enumerate}
        \item[$\blacktriangleright$] 
        \noindent{\ttfamily PROCEDURE:}
        The algorithm proceeds as follows:
        \begin{enumerate}[label=\Alph*]
            \item  \label{singlefirststep}
            Use Theorem \ref{thm: wall crossing} a total of $m-1$ times to obtain the expression:
            \[
                \hspace{2em}\int_{\Mhatmc{m}} \Phi(\mathcal{E}) 
                = 
                \int_{\Mhatmc{0}}\Phi(\mathcal{E})
                +
                \sum_{\tilde{m}=0}^{m-1}
                \sum_{j>0}
                \int_{\Mhat^{\tilde{m}}(\chat-j\eCH_0+j\tilde{m}[\pt])}\Phi_{\tilde{m},j}(\mathcal{E}).
            \]
            Here, we also used $\epsilon_m = \epsilon_0 - m[\pt]$ from the \hyperlink{sec: basic notation}{Basic Notation}. 
            \item\label{singlesecstep}
            If $\ell=0$, then we apply Proposition \ref{prop: stable as m-stable} to the first term on the right-hand side of step \ref{singlefirststep}. The second term on the right-hand side of step \ref{singlefirststep} becomes the collection indexed by $I$. We apply Lemma \ref{lem: Twisting by O(C)} for each term if necessary to obtain moduli spaces of the desired form. More precisely, for the term in the sum corresponding to indices $\tilde{m}$ and $j$, we apply the lemma for the twist by $\mathcal{O}(\lfloor j/r\rfloor C)$. This completes the algorithm in the case $\ell=0$. 
            \item \label{singlethirdstep}
            If $0 < \ell < r$, then we twist the first term on the right-hand side of step \ref{singlefirststep} by $\mathcal{O}(C)$ and apply Lemma \ref{lem: Twisting by O(C)}. The resulting term has Chern character $\chat + r \eCH_0 + \ell [\pt]$ and  satisfies the conditions of Proposition \ref{prop: Gr bundle formula} with $k = r-\ell$. Applying Proposition \ref{prop: Gr bundle formula} strictly decreases the virtual dimension. The term $\int_{\Mhatmc{0}} \Phi(\mathcal{E})$ from step \ref{singlefirststep} becomes
            \[
                \int_{\Mhat^1(p^*c + (n+\ell)[\pt])}\Phi''(\mathcal{E}). 
            \]
            This, and the remaining terms from step \ref{singlefirststep}, become the collection indexed by $I$ (observing that $\Phi' =0$ in this case). As detailed in Step \ref{singlesecstep}, we apply Lemma \ref{lem: Twisting by O(C)} for various twists to bring the remaining terms in $I$ into the desired form. This completes the algorithm.  
        \end{enumerate}
    \end{itemize}
\end{alg}

\begin{alg}[\textit{Nakajima--Yoshioka Blow-Up Algorithm}]\label{alg: NY blow-up (full)}
    Let $\ell \geq 0$ be an integer
    and suppose
    $\chat :=p^*c-\ell \eCH_0$.
    \begin{itemize}[leftmargin=2em]
        \item[$\blacktriangleright$]
        \noindent{\ttfamily INPUT:} The input  for the algorithm is an integral of the form
        \[
            \int_{\Mhatmc{m}} \Phi(\mathcal{E}).
        \]
        \item[$\blacktriangleright$]
        \noindent{\ttfamily OUTPUT:} The output is an identity of the form
        \[
            \int_{\Mhatmc{m}}\Phi(\mathcal{E}) 
            = 
            \sum_{n\geq 0} \int_{M(c+n[\pt])}\Phi'_n(\mathcal{E}).
        \]
        \item[$\blacktriangleright$]
        \noindent{\ttfamily PROCEDURE:}
        The algorithm proceeds as follows:
        \begin{enumerate}[label=\Alph*]
            \item 
            We can assume $0\leq \ell <r$ by applying Lemma \ref{lem: Twisting by O(C)} an appropriate number of times. Apply  Algorithm \ref{alg: NY blow-up single term} to obtain
            \[
                \int_{\Mhatmc{m}}\Phi(\mathcal{E})
                = 
                \int_{M(c)}\Phi'_0(\mathcal{E}) 
                + 
                \sum_{i\in I} 
                \int_{M_i}\Phi_i(\mathcal{E}).
        \]
            \item \label{totalstepii}
             Apply  Algorithm \ref{alg: NY blow-up single term} to those terms indexed by $i\in I$ for which $M_i$ has maximal virtual dimension, and remove those terms from $I$. Collect the resulting output terms of type (i) to form $\Phi_n'$ for the appropriate $n$ (since the virtual dimensions of the spaces appearing in the algorithm are bounded above by $M(c)$, one automatically has $n\geq 0$). Collect the resulting output terms of type (ii) of lower virtual dimension and add those terms to $I$.  We now use $I$ to denote the indexing set for this new collection. 
             \item 
             Iterate Step \ref{totalstepii} until $I$ indexes the empty set. This will terminate, because after each step, $I$ is finite and the maximal virtual dimension of the moduli spaces indexed by $I$ decreases at each step.  
        \end{enumerate}
    \end{itemize}
    
\end{alg}

\section{Weak structure theorem}\label{sec: weak structure thm} 
Here we explain a refined version of Nakajima--Yoshioka's blowup formalism which holds when integrating a multiplicative class of the tangent bundle.\footnote{Recall that a characteristic class $\Theta:K(-)\to A^*(-)$ is multiplicative if it satisfies the identity  $\Theta(F+E) = \Theta(F)\cdot \Theta(E)$ for all choices of $F$ and $E$.} This gives a unified version of Nakajima--Yoshioka's \cite[Thm. 2.6]{NY_PCSOB3} and Kuhn--Tanaka's \cite[Thm. 1.5]{KT_Blowup}. In particular, we give a concise presentation of how the two situations are related, which is only implicit in their work. We begin this section by defining some useful notation.

\begin{notation}
    Denote by $\mathcal{S}$ the collection of symbols
    \[
        \Big\{\,\,
        c_i\big(\relhom(C_m,-)\big)
        \,,~
        c_i\big(\relhom(-,C_m)\big)
        \,\,\Big\}_{i\in\ZZ_{>0}, m\in \ZZ}
    \]
    and $\mathfrak{R} := \coeffs\psl \mathcal{S} \psr\psl \varepsilon_1,\varepsilon_2\psr$ the ring of power series in the symbols from $\mathcal{S}$ and in variables $\varepsilon_1,\varepsilon_2$.
    For an element $P\in \mathfrak{R}$ and for $\mathcal{F}$ a sheaf on $M\times \Xhat$, we denote by $P(\mathcal{F})$ the Chow class in $A^*(M)$ obtained by evaluating the symbols of $\mathcal{S}$ at $\mathcal{F}$ and by setting $\varepsilon_i = c_1(\tT_i)$ in Situation \ref{sit: framed main}, respectively $\varepsilon_{i}=0$ in Situation \ref{sit: H-stable}. 
\end{notation}

We now state the main result of this section. 

\begin{thm}[Weak Structure Theorem] 
    \label{thm: omega_j structure theorem}
	Let $k\geq 0$ be an integer, $\Theta$ a multiplicative class and $\nu_1, \ldots, \nu_r$ be formal variables.  	
	There exists a unique collection of power series $\{ \,\Omega_n \}_{n\geq 0}$ in $\coeffs\psl\varepsilon_1,\varepsilon_2, \nu_1,\ldots,\nu_r\psr$ which satisfies the following condition:
    \begin{flushright}
        \begin{tabular}{p{0.00\linewidth} | p{0.9\linewidth} @{}}
            \hspace{1em}&        Whenever we are in Situation \ref{sit: H-stable} or \ref{sit: framed main} with $\chat$ satisfying $(\chat_1,[C]) =k$, then we have\\[0.5em]
            &
            \hspace{2.5em}
            $\displaystyle
                   \int_{\Mhatc} \Theta\big(\Tvir\big)  = 
    	 	\sum_{n\geq 0} \int_{M(c+n\,[\pt])} \Theta\big(\Tvir\big) \, \Omega_n(\mathcal{E}),
            $\\
            &
        	where $\Omega_n(\mathcal{E})$ denotes the Chow-cohomology class obtained by replacing the variable $\nu_i$ by the class $c_i(\relhom(\mathcal{O}_{\pt},\mathcal{E}))$.
        \end{tabular}
    \end{flushright}
    In fact, the collection $\{\Omega_n\}_{n\geq 0}$ is uniquely determined if one only assumes this condition to hold in Situation \ref{sit: framed main}.
\end{thm}

\begin{rem}
More generally, the existence of Theorem \ref{thm: omega_j structure theorem} holds (with the same proof) if instead of just a multiplicative class of the tangent bundle, one considers a product $\Theta(T^{\vir})\Phi(\mathcal{E})$, where $\Theta$ is multiplicative, and $\Phi$ is as in Notation \ref{not: phi notation}, and is also multiplicative in its input, i.e. $\Phi(\mathcal{F}_1\oplus \mathcal{F}_2)=\Phi(\mathcal{F}_1)\cdot \Phi(\mathcal{F}_2)$. Here, $\Phi$ may only be defined for a particular choice of $X$ or for a particular situation. The uniqueness statement holds (with the same proof) whenever $\Phi$ is defined in a way that makes sense in Situation \ref{sit: framed main}.  
\end{rem}

The idea of the proof is to follow a refined version of the Nakajima--Yoshioka Blow-Up Algorithm (Algorithm \ref{alg: NY blow-up (full)}) which keeps track of extra structure related to the coefficients $\Phi_{i}$. 
We first collect some consequences from Theorem \ref{thm: wall crossing} and Proposition \ref{prop: Gr bundle formula} when applied to multiplicative classes.

\begin{lem}[Series version of Nakajima--Yoshioka-Algorithm components]
\label{lem: algo steps multiplicative}
For $P\in \mathfrak{R}$ the components of Algorithm \ref{alg: NY blow-up (full)} (the Nakajima--Yoshioka Blow-Up Algorithm) define canonical series in $\mathfrak{R}$ in the following ways:
\begin{enumerate}[label = (\roman*)]
    \item \emph{Wall-Crossing Formula (Theorem \ref{thm: wall crossing}):} For integers $k,m\geq 0$, there exists a canonical collection of power series $\{Q_j\}_{j\geq 1}$ in $\mathfrak{R}$ which satisfies the following condition:
    \begin{flushright}
        \begin{tabular}{p{0.00\linewidth} | p{0.95\linewidth} @{}}
            \hspace{1em}&        Whenever we are in Situation \ref{sit: H-stable} or \ref{sit: framed main}  with $\chat$ satisfying $(\chat_1,[C]) =k$, then we have\\[0.5em]
            &
            \hspace{2.5em}
            $\displaystyle
                   \int_{\Mhatmc{m+1}}\Theta(\Tvir)P(\mathcal{E}) - \int_{\Mhatmc{m}}\Theta(\Tvir)P(\mathcal{E})
            $ \\[1em]
            &
            \hspace{7em}
            $\displaystyle
                   =~\sum_{j\geq 1}\int_{\Mhat^m(\chat-j\eCH_m)}\Theta(\Tvir)Q_j(\mathcal{E}).
            $ 
        \end{tabular}
    \end{flushright}
    \item \label{algstepmultii}
    \emph{Twisting by $\mathcal{O}(C)$ (Lemma \ref{lem: Twisting by O(C)}):}
    There exists a canonical power series $P^{\dagger}\in \mathfrak{R}$ satisfying $P^{\dagger}(\mathcal{E}) = P(\mathcal{E}(C))$ whenever we are in Situation \ref{sit: H-stable} or \ref{sit: framed main} with arbitrary $\chat$. 
    \item
    \emph{$\mathrm{Gr}$-Bundle Formula (Proposition \ref{prop: Gr bundle formula}):}
    For each integer $0<k< r$ there exists a canonical power series $Q$ in $\mathfrak{R}$, which satisfies the condition:
    \begin{flushright}
        \begin{tabular}{p{0.00\linewidth} | p{0.95\linewidth} @{}}
            \hspace{1em}&        Whenever we are in Situation \ref{sit: H-stable} or \ref{sit: framed main}  with $\chat$ satisfying $(\chat_1,[C])=-k$, then we have \\[0.5em]
            &
            \hspace{2.5em}
            $\displaystyle
                   \int_{\Mhatmc{1}}\Theta(\Tvir)P(\mathcal{E}) = \int_{\Mhat^1(\chat-k\eCH_0)}\Theta(\Tvir)Q(\mathcal{E}).
            $ 

        \end{tabular}
    \end{flushright}

\end{enumerate}
\end{lem}

\begin{proof}
For part \ref{algstepmultii}, one obtains $P^{\dagger}$ by replacing every occurrence of $\relhom(C_m, -)$ and $\relhom(-,C_m)$ in $P$ by $\relhom(C_{m-1},-)$ and $\relhom(-,C_{m-1})$ respectively. 
For parts (i) and (iii), we apply Theorem \ref{thm: wall crossing} and Proposition \ref{prop: Gr bundle formula} respectively to the class $\Phi(\mathcal{F}):= \Theta(\Tvir_{\mathcal{F}}) P(\mathcal{F})$ (where $\Tvir_{\mathcal{F}}$ is defined in the \hyperlink{sec: basic notation}{Basic Notation}). The description of the virtual tangent bundle from Remark \ref{rem: extending integral to m-stable} and the discussion preceding Notation \ref{not: integral notation} shows that 
$\Phi\big(\mathcal{E} 
\oplus \bigoplus_{i=1}^j 
C_m\eLB^{-h_i}\big)$
can be written as:
\begin{align*}
    &
     \Theta(\Tvir)
    \,
    P\Big(\mathcal{E}\oplus \bigoplus_{i=1}^j C_m\eLB^{-h_i}\Big) 
    \prod_{1\leq i_1,i_2\leq j}
    \Theta\big(\relhom(C_m,C_m)\eLB^{h_{i_2}-h_{i_1}}\big)
    \\
    &~
    ~\cdot~\prod_{i=1}^j \Theta\big(\relhom(\mathcal{E},C_m)\eLB^{-h_i}\big)\, \Theta\big(\relhom(C_m,\mathcal{E}) \eLB^{h_i}\big).
\end{align*}
We can make use of Lemma \ref{lem: RHom properties} \ref{RHomprop4} and the basic properties of Chern classes to rewrite this expression canonically as an expression in variables $h_i$, elements of $\mathcal{S}$ and in $\varepsilon_1,\varepsilon_2$. 
It follows that the expressions appearing on
 the right-hand side of the equations Theorem \ref{thm: wall crossing} and Proposition \ref{prop: Gr bundle formula} are of the desired form. This implies parts (i) and (iii).
\end{proof}

We can now state a refined version of Algorithm \ref{alg: NY blow-up single term}.

\begin{alg}[\emph{Refined Blow-Up Algorithm For Single Terms}]\label{alg: NY blow-up single term (refined)}
    Let $n,\ell\geq 0$ be integers
    and set 
    $\chat :=p^*c-\ell \eCH_0 + n[\pt]$.
    \begin{itemize}[leftmargin=2em]
        \item[$\blacktriangleright$]
        \noindent{\ttfamily INPUT:} The input  for the refined algorithm is an expression of the form
        \[
            \int_{\Mhatmc{m}} \Theta(\Tvir)P(\mathcal{E}),
        \]
        where $P\in \mathfrak{R}$. 
        \item[$\blacktriangleright$]
        \noindent{\ttfamily OUTPUT:} The output is an identity of the form
        \[
                \int_{\Mhatmc{m}}\Theta(\Tvir)P(\mathcal{E})
                = 
                \int_{M(c+n[\pt])}\Theta(\Tvir)\Omega(\mathcal{E}) 
                + 
                \sum_{i\in I} 
                \int_{M_i}\Theta(\Tvir)P_i(\mathcal{E}),
        \]
        where
        \begin{enumerate}
            \item
            $\Omega(\mathcal{E})$ is a power series in the symbols $c_i(\relhom(\mathcal{O}_{\pt},\mathcal{E}))$, and is zero if $0<\ell<r$. 
            \item
            $I$ is a finite collection indexing tuples $(m_i,\ell_i, d_i,P_i)$, where $m_i,d_i,\ell_i\geq 0$, and where $P_i\in \mathfrak{R}$. Moreover, for $M_i:=\Mhat^{m_i}(b_i)$ with $b_i=p^*c-\ell_i \eCH_0 + (n+d_i)[\pt]$ they satisfy  
                $
                    \vdim\,\Mhat^{m_i}(b_i) 
                    ~<~
                    \vdim\,\Mhatmc{m}.
                $
        \end{enumerate}
        \item[$\blacktriangleright$]
        \noindent{\ttfamily PROCEDURE:}
        The algorithm proceeds as follows:
        \begin{enumerate}[label=\Alph*]
            \item
            Follow Algorithm \ref{alg: NY blow-up single term} with the Wall-Crossing Component (Theorem \ref{thm: wall crossing}), the Twisting Component (Lemma \ref{lem: Twisting by O(C)}) and the $\mathrm{Gr}$-Bundle Component (Proposition \ref{prop: Gr bundle formula}) replaced by their counterparts from Lemma \ref{lem: algo steps multiplicative}.
            \item
             If $0< \ell < r$, then the algorithm is complete. If $\ell=0$, then consider the term 
            \[
                \int_{M(c+n[\pt])}\Phi(\mathcal{E})
                = \int_{\Mhatmc{0}} \Theta(\Tvir)P(\mathcal{E})
            \]
            arising in step (i) of Algorithm \ref{alg: NY blow-up single term} and apply (to this term) the following procedure:
            
            \begin{enumerate}
                \item
                Use the formulas from Lemma \ref{lem: RHom properties} \ref{RHomprop3} to rewrite all Chern classes $c_i(\relhom(\mathcal{E},C_m))$ and $c_i(\relhom(C_m, \mathcal{E}))$ appearing in $P$ in terms of only the Chern classes $c_i(-\relhom(C_0,\mathcal{E}))$ for $i\geq 1$ and of $\varepsilon_a = c_1(\tT_a)$ for $a=1,2$ (the sign introduced here is important to obtain a vector bundle for Step (c)). Denote the resulting power series by $P'$.
                \item
                Use the Wall-Crossing Formula from Lemma \ref{lem: algo steps multiplicative} to get
                \begin{align*}
                    \hspace{5em}
                    &\int_{M(c+n[\pt])}\Phi(\mathcal{E})
                    \\ 
                    &\hspace{2em}=\int_{\Mhatmc{1}}\Theta(\Tvir)P'(\mathcal{E}) - \sum_{j\geq 1} \int_{\Mhat^0(\chat-j\eCH_0)}\Theta(\Tvir)Q_j(\mathcal{E}).
                \end{align*}
                for some $Q_j\in \mathfrak{R}$.
                \item
                By Lemma \ref{lem: RHom properties} \ref{RHomprop2}, we have that $-\relhom(C_0,\mathcal{E})$ is the $K$-theory class of a rank $r$ vector bundle on $\Mhatmc{1}$. Therefore, set all Chern classes of degree higher than $r$ equal to zero in $P'$ to get a power series $\Omega\in \ZZ\psl\varepsilon_1,\varepsilon_2,\nu_1,\ldots,\nu_r\psr$ such that
                \[\int_{\Mhatmc{1}}\Theta(\Tvir)P'(\mathcal{E}) =  \int_{\Mhatmc{1}}\Theta(\Tvir)\Omega(\mathcal{E}).\]
                \item
                Use the Wall-Crossing Formula from Lemma \ref{lem: algo steps multiplicative} on the term from Step (c) to get the identity
                \begin{align*}
                    \hspace{6.5em}
                    &\int_{M(c+n[\pt])}\Phi(\mathcal{E})
                    \\ 
                    &\hspace{0.8em}=\int_{\Mhatmc{0}}\Theta(\Tvir)\Omega(\mathcal{E}) + \sum_{j\geq 1} \int_{\Mhat^0(\chat-j\eCH_0)}\Theta(\Tvir)\Big(Q'_j(\mathcal{E})-Q_j(\mathcal{E})\Big)
                \end{align*}
                for canonical power series $Q_j'$.
                \item
                Apply Proposition \ref{prop: stable as m-stable} to the first term on the right-hand side of Step (d). The second term on the right-hand side of Step (d) is absorbed into the collection indexed by $I$. This completes the algorithm in the case $\ell=0$.
            \end{enumerate}
        \end{enumerate}

    \end{itemize}
\end{alg}

\begin{rem}\label{rem:algindep}
It follows from the algorithm that the output depends only on $\Theta,P,\ell$ and $m$, but is independent of $c$ and $n$, and on whether we are in Situation \ref{sit: H-stable} or \ref{sit: framed main}, in the following sense: The series $\Omega$ depends only on $\Theta, P, \ell$ and $m$. Moreover, there is an infinite collection $\overline{I}$ of tuples $(m_i,\ell_i,d_i,P_i)$ which depends only on $\Theta, P, \ell$ and $m$, such that: $I$ is exactly the subset of $\overline{I}$ consisting of those tuples for which $\vdim\, \Mhat^{m_i}(p^*c-\ell_i\eCH_0+(n+d_i)[\pt])$ is non-negative. 
\end{rem}

\begin{proof}[{Proof of Theorem \ref{thm: omega_j structure theorem}}]
\emph{Existence of $\{\Omega_n\}_{n\geq 0}$:}
We have a refined version of Algorithm \ref{alg: NY blow-up (full)} (the Nakajima--Yoshioka Blow-Up Algorithm), obtained by replacing 
Algorithm \ref{alg: NY blow-up single term} by the refined version, Algorithm \ref{alg: NY blow-up single term (refined)}, and by replacing Lemma \ref{lem: Twisting by O(C)} by its counterpart in Lemma \ref{lem: algo steps multiplicative}. Then, this modified algorithm applies to any $\Phi(\mathcal{E})=\Theta(T^{\vir})P(\mathcal{E})$ with $P\in \mathfrak{R}$, whenever we are in Situation \ref{sit: H-stable} or \ref{sit: framed main}  with $\chat$ such that $(\chat_1,[C]) =k\geq 0$. It shows that there are classes $\Omega_n(\mathcal{E})$, which are power series in $c_i(\relhom(\mathcal{O}_{\pt},\mathcal{E}))$ and that satisfy the equation
\begin{align}
    \label{eq: Gamma defining eq}
    \int_{\Mhatc} \Theta\big(\Tvir\big)  = 
 	\sum_{n\geq 0} \int_{M(c+n\,[\pt])} \Theta\big(\Tvir\big) \, \Omega_n(\mathcal{E}). 
\end{align}

We observe that the algorithm creates the same collection of formal power series $\{\Omega_n\}_{n\geq 0} \subset \coeffs \psl \sf \varepsilon_1,\varepsilon_2,\nu_1,\ldots,\nu_r\psr$ for any choice of data from Situation \ref{sit: H-stable} or \ref{sit: framed main}  with $\chat$ such that $(\chat_1,[C]) =k$. This is due to the independence in the single step asserted in Remark \ref{rem:algindep}.

\vspace{0.5em}
\noindent\emph{Uniqueness of $\{\Omega_n\}_{n\geq 0}$:}
    We prove that any collection of power series $\{\Omega_n\}$ satisfying the condition in the statement of Theorem \ref{thm: omega_j structure theorem} must be unique. As mentioned there, the following stronger statement is true: 
    \begin{flushright}
        \begin{tabular}{p{0.00\linewidth} | p{0.95\linewidth} @{}}
            \hspace{1em}& 
            The collection $\{\Omega_n\}$ is the unique collection which satisfies the requirement that \eqref{eq: Gamma defining eq} holds \emph{in Situation \ref{sit: framed main}} for all  $\chat = r- k\eCH_0 + n_0[\pt]$ with $n_0\in \ZZ_{\geq 0}$. 
        \end{tabular}
    \end{flushright}
    Hence, for the rest of the proof we will assume we are in Situation \ref{sit: framed main} and that we have some collection $\{\Omega_n\}$ for which \eqref{eq: Gamma defining eq} holds for all  $\chat = r- k\eCH_0 + n_0[\pt]$ with $n_0\in \ZZ_{\geq 0}$. Now, for any $n_0\in \ZZ_{\geq 0}$, equation \eqref{eq: Gamma defining eq} can be rearranged as
    \begin{equation}\label{eq: Gamma recursive}
    \begin{aligned}
        &\int_{M(r-0[\pt])}\Theta(\Tvir)\Omega_{n_0}(\mathcal{E})\\
        & \hspace{2em}
        =
        \int_{\Mhat(r-k\eCH_0-n_0[\pt])}\Theta(\Tvir)
        -
        \sum_{n=0}^{n_0-1}\int_{M(r-(n_0-n)[\pt])}\Theta(\Tvir) \, \Omega_n(\mathcal{E}). 
    \end{aligned}
    \end{equation}
    As shown, for example, by Nakajima--Yoshioka in \cite[Prop. 2.9 \& Thm 2.11]{NY_ICOB1}, the moduli space $M(r+0[\pt])$ is a point with tangent bundle $\Tvir_{M(r+0[\pt])}=0$. Now, letting $\tT_1, \tT_2, \eT_1,\ldots,\eT_r$ be the characters of the (trivial) $\widetilde{T}$-action on $M(r+0[\pt])$, the universal sheaf of $M(r+0[\pt])$ is $\mathcal{E}_0=\bigoplus_{i=1}^r \mathcal{O}_{\PP^2}\,\eT_i$.\footnote{The $\torus$-action on the moduli spaces is described further in Definition \ref{def: torus action}.}  
    In particular, one can compute that $\relhom(\mathcal{O}_{\pt},\mathcal{E}_0) =\oplus_{i=1}^r \eT_i$ and we have the following equation in $A^*_{\torus}(\pt)$:
    \[
        \Omega_{n_0}\left(\bigoplus_{i=1}^r \mathcal{O}_{\PP^2}\,\eT_i\right)
        =
        \int_{\Mhat(r-k\eCH_0-n_0[\pt])}\Theta(\Tvir)
        -
        \sum_{n=0}^{n_0-1}\int_{M(r-(n_0-n)[\pt])}\Theta(\Tvir) \, \Omega_n(\mathcal{E}). 
    \]
    
    We observe that the left hand side is by definition a power series in the terms 
    $c_i(\relhom(\mathcal{O}_{\pt},\mathcal{E})) = \sigma_i(c_1(\eT_1),\ldots,c_1(\eT_r))$ for $1\leq i\leq r$, where $\sigma_i$ denotes the $i$-th elementary symmetric polynomial.
    These are analytically independent (since the terms $c_1(\eT_i)$ are), so the equality \eqref{eq: Gamma recursive} determines $\Omega_{n_0}$ inductively from the $\Omega_n$ for $n=0,\ldots,  n_0-1$. (Note that this includes the base case of the induction.) Thus the collection of all the $\Omega_n$ is uniquely determined. 
\end{proof}

\section{Proof of the blowup formula for the \texorpdfstring{$\chi_y$}{χy}-genus} \label{sec: equivariant proof}

In this section, we prove our main theorem, Theorem \ref{thm: main theorem}. For this, we apply Theorem \ref{thm: omega_j structure theorem} to the class $\Theta_y$ giving rise to the $\chi_y$-genus. Then, it is clearly enough to show that each power series $\Omega_n$ is constant and equal to the coefficient of $q^n$ in the power series $\mathsf{Y}_k$ appearing in Theorem \ref{thm: main theorem}. 
To do this, we use the fact that by \eqref{eq: Gamma recursive}, the $\Omega_n$ can be uniquely determined from the equivariant theory of framed moduli spaces in Situation \ref{sit: framed main}.

\begin{rem}
	An easier version of the proof using the total Chern class $\Theta(-)= c(-)$ in place of $\Theta_y$ gives a direct proof of the blowup formula for virtual Euler characteristics. 
\end{rem}

In the rest of this section, we work in Situation $B$ and we let $\{\Omega_n\}$ denote the collection of power series obtained from Theorem \ref{thm: omega_j structure theorem} by the choice $\Theta=\Theta_y$. Note that, by definition, these are power series with coefficients in $\coeffs= \QQ(y)$.
We recall the $\torus$-action on the moduli spaces $\Mc$ and $\Mhatc$ from \cite{NY_ICOB1}.

\begin{definition}[\textit{$\torus$-Action}]
\label{def: torus action}
\cite[\S 2 -- \S 3]{NY_ICOB1}
Let $M$ be one of $\Mc$ or $\Mhatc$ and let $(E,\varphi) \in M$ be a framed sheaf. For $(t_1,t_2,e_1,\ldots,e_r)\in \torus \cong (\CC^*)^{2+r}$ we define the following:
\begin{enumerate}
    \item 
    In the case of $M = \Mc$: Let $F_{t_1,t_2}$ denote the automorphism of $\PP^2$ defined via $[z_0:z_1:z_2] \mapsto [z_0 : t_1z_1:t_2z_2]$. In the case of $M=\Mhatc$: We denote again by $F_{t_1,t_2}$ the unique lift of this automorphism to $\widehat{\PP}^2$.
    \item
    In either case, we let $\varphi_{(t_1,t_2,e_1,\ldots,e_r)}$ denote the composition
    \[
        (F_{t_1,t_2}^{-1})^*E|_{\ell_\infty}\overset{(F_{t_1,t_2}^{-1})^*\varphi}{\longrightarrow} (F_{t_1,t_2}^{-1})^*\mathcal{O}_{\ell_{\infty}}^{\oplus r}
        \overset{\sim}{\longrightarrow}
        \mathcal{O}_{\ell_{\infty}}^{\oplus r}
        \overset{\sim}{\longrightarrow}
        \mathcal{O}_{\ell_{\infty}}^{\oplus r}
    \]
    where the middle arrow comes from the $(\CC^*)^2$-action on $\PP^2$ (respectively $\PPhat^2$), and the last arrow is $(x_1,\ldots,x_r)\mapsto (e_1x_1,\ldots,e_r x_r)$. 
\end{enumerate}
The action of $\torus$ on $M$ is defined by
\[
    (t_1,t_2,e_1,\ldots,e_r)\cdot(E,\varphi) 
    = 
    \big(\,(F_{t_1,t_2}^{-1})^*E,~ \varphi_{(t_1,t_2,e_1,\ldots,e_r)}\, \big). 
\]
\end{definition}

\begin{notation}[\textit{$\torus$-Representations}]\label{not: T representations}
We denote by $\tT_i$ and $\eT_a$, the one dimensional $\torus$-modules given respectively by
\[
    (t_1,t_2,e_1,\ldots,e_r)\mapsto t_i
    \hspace{1cm}\mbox{and}\hspace{1cm}
    (t_1,t_2,e_1,\ldots,e_r)\mapsto e_a.
\]
\end{notation}

We now recall the description of the torus fixed points of the moduli spaces of framed sheaves on $\PP^2$ and $\PPhat^2$ given in \cite{NY_ICOB1}. We also recall their description of the equivariant $K$-theory class given by the restriction of the tangent bundle to each fixed point. The results are as follows:

\begin{thm}\cite[Prop. 2.9 \& Thm. 2.11]{NY_ICOB1} \label{thm: fixed point and tangent spaces}
Let $c = r - n[\pt]$ for $n\geq 0$ an integer. Then we have:
\begin{enumerate}
	\item The fixed points of $M(c)$ are in bijection to the set of $r$-tuples of Young diagrams $\bm{Y}=(Y_1,\ldots,Y_r)$ satisfying $\sum_{i=1}^r\abs{Y_i} = n$.	
	\item Let $\bm{Y} =(Y_1,\ldots,Y_r)$ correspond to a fixed point of $M(c)$. The equivariant $K$-theory class of the restricted tangent bundle $\T_{M(c)}|_{\bm{Y}}$ is given by 
    \[
    	\sum_{a,b = 1}^ r \N^{\bm{Y}}_{a,b}(\tT_1,\tT_2),
    \]
    where $\N^{\bm{Y}}_{a,b}(\tT_1,\tT_2)$ is given by
    \[
        \eT_{b}\,\eT_{a}^{-1}\times \left\{\sum_{s\in Y_{a}}\left(\tT_1^{-l_{Y_{b}}(s)}\tT_2^{a_{Y_{a}}(s)+1}\right) + \sum_{t\in Y_{b}} \left(\tT_1^{l_{Y_{a}}(t)+1} \tT_2^{-a_{Y_{b}}(t) }\right) \right\}.
    \]
    Here, $a_Y (i,j) := \lambda_i -j$ and  $l_Y (i,j) :=  \lambda^t_j -i$ are the \emph{arm length} and \emph{leg length} of the box $(i,j)$ in $Y$, where $\lambda$ is the partition associated to $Y$ with transpose $\lambda^t$. 
\end{enumerate}
\end{thm}

\begin{thm}\cite[Prop. 3.2 \& Thm. 3.4]{NY_ICOB1} \label{thm: fixed point and tangent spaces blowup}
    Let $k,n\in \ZZ$ be integers and $\eCH_0 = [C]-\frac{1}{2}[\pt]$. If $\chat = r + k\eCH_0 - n [\pt]$, then we have:
	\begin{enumerate}
		\item The fixed points of $\Mhatc$ are in natural bijection to the set of triples $(\bm{Y},\bm{Z},\bm{k})$, where  $\bm{Y}=(Y_1,\ldots,Y_r)$ and $\bm{Z}=(Z_1,\ldots,Z_r)$ are $r$-tuples of Young diagrams and $\bm{k}=(k_1,\ldots,k_r)\in \ZZ^r$ such that 
		\[
		    \hspace{4em}\sum k_i = k,
		    \hspace{1em}\mbox{and}\hspace{1em}
		    \sum_{j=1}^r \big(\abs{Y_i} +\abs{Z_i}\big)+ \frac{1}{2r}\sum_{1\leq i<j\leq r}(k_i-k_j)^2 = n + \frac{k(r-k)}{2r}.
		\]
		\item Let $F$ be the fixed point of $\Mhatc$ given by $(\bm{Y},\bm{Z},\bm{k})$. The equivariant $K$-theory class of the restricted tangent bundle $\T_{\Mhatc}|_{F}$ is given by 
		\[
    		\sum_{a,b = 1}^r \mathrm{L}_{a,b}(\tT_1,\tT_2) 
    		+
    		\tT_1^{k_{b}-k_{a}} \N^{\bm{Y}}_{a,b}(\tT_1,\tT_2/\tT_1) 
    		+
    		\tT_2^{k_{b}-k_{a}}\N^{\bm{Z}}_{a,b}(\tT_1/\tT_2,\tT_2),
		\]
		where the expressions $N_{a,b}$ are defined as in Theorem \ref{thm: fixed point and tangent spaces} and 
		\[
            \mathrm{L}_{a,b}(\tT_1,\tT_2) = \eT_{b}\,\eT_{a}^{-1}\times \begin{dcases}\sum_{\substack{i,j\geq 0 \\ i+j\leq k_{a}-k_{b}-1}} \tT_1^{-i}\tT_2^{-j}, & \mbox{ if } k_{a} > k_{b};\\
			\sum_{\substack{i,j\geq 0 \\ i+j\leq k_{b}-k_{a}-2}} \tT_1^{i+1}\tT_2^{j+1}, & \mbox{ if } k_{a} +1 < k_{b};\\
			0, & \mbox { otherwise.}
		    \end{dcases}
		\]
	\end{enumerate}
\end{thm}

\begin{ex}\label{rem: moduli spaces base cases}
    From the dimension formula $c_1^2-2r\ch_2$, one can see that the spaces $M(r+0)$ and $\Mhat(r+0)$ consist of a collection of reduced points, while the spaces $\Mhat(r-k\eCH_0)$ are empty for $k>0$. From Theorems \ref{thm: fixed point and tangent spaces} and \ref{thm: fixed point and tangent spaces blowup}, it follows that in fact $M(r+0)$ and $\Mhat(r+0)$ consist of a single reduced point. Moreover, the universal sheaf $\mathcal{E}_0$ on $\PP^2\times M(r)\simeq\PP^2$ is isomorphic to a direct sum $\bigoplus_{i=1}^r\mathcal{O}_{\PP^2}\eT_i$. It follows that $\relhom(\mathcal{O}_{\pt},\mathcal{E}
    _0)\simeq \bigoplus_{i=1}^r \eT_i$.
\end{ex}
\begin{prop}
\label{prop: Gamma's are constant}
The collection $\{\Omega_n\}$ obtained from applying Theorem \ref{thm: omega_j structure theorem} (the Weak Structure Theorem) to $\Theta = \Theta_y$ is a collection of elements from $\coeffs\psl\varepsilon_1,\varepsilon_2\psr$, where $\coeffs=\QQ(y)$. 
\end{prop}
\begin{proof}
We prove by induction on $n$ that the $\Omega_n$ are constant in $\nu_1,\ldots,\nu_r$. For this, we use the inductive formula \eqref{eq: Gamma recursive}. For $n_0=0$, we get that 
\[\Omega_0 = \int_{\Mhat(r-k\eCH_0)}\Theta(T^{\vir}).\]
By Example \ref{rem: moduli spaces base cases}, the right hand side is equal to zero if $k\neq 0$, and equal to one if $k=0$ , therefore $\Omega_0$ is constant in either case. 
Now let $n_0>0$ be fixed, and suppose we know that $\Omega_n$ lies in $\coeffs\psl\varepsilon_1,\varepsilon_2\psr$ for each $n<n_0$. 
Then \eqref{eq: Gamma recursive} becomes 
\begin{equation}\label{eq:Sigma constancy}
	\Omega_{n_0}\left(\bigoplus_{i=1}^r \mathcal{O}_{\PP^2}\,\eT_i\right)
	=
	\int_{\Mhat(r-n_0[\pt]-k\epsilon_0)}\Theta(\T)- \sum_{n=0}^{n_0-1}\Omega_n \int_{M(r+(n-n_0)[\pt])}\Theta(\T).
\end{equation}
By Example \ref{rem: moduli spaces base cases}, we know that the left hand side is a power series in the variables $\{c_1(\eT_1),\ldots, c_1(\eT_r), c_1(\tT_1),c_1(\tT_2)\}$ which is symmetric in the $c_1(\eT_i)$ variables. We may equivalently write it as a power series in the variables $d_i := \ch(\eT_i)-1$ and $s_i:= \ch(\tT_i)-1$ via the substitutions $c_1(\eT_i) = \operatorname{log}(1+d_i)$ and $c_1(\tT_i) = \operatorname{log}(1+s_i)$, where we use the formal power series expansion of $\operatorname{log}(1+x)$. 

We more closely examine the right hand side of \eqref{eq:Sigma constancy}. By definition of integration as a sum over fixed points, 
we can rewrite each integral appearing on the right hand side as a sum 
\[\sum_G \frac{\Theta(\T|_{G})}{\Eu(\T|_G)},\]
where $G$ ranges over the fixed points of a given moduli space. 
Note that for a nontrivial $\torus$-equivariant line bundle $L$ with $\alpha = c_1(L)$, we have
\[\frac{\Theta_y(L)}{\Eu(L)} = \frac{1-ye^{-\alpha}}{1-e^{-\alpha}} = \frac{1}{\alpha}\operatorname{td}(\alpha) (1-ye^{-\alpha}).\]
In particular, this fraction is well-defined. Write $e_i:=\ch(\eT_i)$ and $t_i:=\ch(\tT_i)$.
Using Theorems \ref{thm: fixed point and tangent spaces} and \ref{thm: fixed point and tangent spaces blowup}, and examining the formulas given there, we find that each integral is a sum of products over terms of the form
\[\frac{1-y(e_a/e_b) t_1^{i} t_2^{j}}{1-(e_a/ e_b) t_1^{i} t_2^{j}},\]
where $a,b\in \{1,\ldots,r\}$, and where $(i,j)\neq (0,0)$ if $a=b$. We may rewrite each such term as a fraction of two coprime elements in $\coeffs[e_1,\ldots,e_r,t_1,t_2]$ by cancelling the potential negative powers of $e_b,t_1$ and $t_2$ appearing in numerator and denominator. After doing this, for each fraction with $a\neq b$ the denominator is now of one of the forms
\begin{equation}\label{eq:appearingtemrs}
e_b-e_at_1^it_2^j,\quad e_bt_1^{-i}-e_at_2^j,\quad e_b t_2^{-j}-e_at_1^i, \mbox{ or}\quad e_b t_1^{-i}t_2^{-j}-e_a,\end{equation}
depending on which of the indices $i,j$ are positive. For $a=b$, the resulting denominator is a polynomial in only $t_1,t_2$. By the Gauss Lemma, each polynomial \eqref{eq:appearingtemrs} is irreducible in $\coeffs[e_1,\ldots,e_r,t_1,t_2]$ if $a\neq b$. Moreover, each such polynomial vanishes at $e_1= \cdots = e_r=t_1=t_2 = 1$. Changing variables to $d_i:=e_i-1$ and $s_i:=t_i-1$, we find that each expression in \eqref{eq:appearingtemrs} with $a\neq b$ gives rise to an irreducible polynomial that vanishes at the origin in $\coeffs[d_1,\ldots,d_r,s_1,s_2]$.

Multiplying both sides of \eqref{eq:Sigma constancy} with all the denominators of the form \eqref{eq:appearingtemrs} appearing in the right hand side, we obtain an equality of the form 
\begin{equation*}
P(d_1,\ldots,d_r,s_1,s_2)\, \Omega_{n_0}(d_1,\ldots,d_r,s_1,s_2) = Q(d_1,\ldots,d_r,s_1,s_2)
\end{equation*}
in the power series ring $\coeffs\psl s_1,s_2,d_1,\ldots,d_r\psr$, where $P$ and $Q$ are polynomials, and where the total degree of $Q$ with respect to the $d_i$ is at most the total degree of $P$ with respect to the $d_i$. Moreover, we may clear the common irreducible polynomial factors of $P$ and $Q$ and assume that they are relatively prime. We claim that then, $P$ and $Q$ are constant in the $d_i$. Indeed, $P$ is a product of terms \eqref{eq:appearingtemrs} and terms not involving the $d_i$, so each irreducible factor of $P$ that involves the $d_i$ satisfies the assumption of Lemma \ref{lem : polydivisible} below. Then, by that lemma, each such factor must also divide $Q$ as a polynomial, contradicting our relatively-prime assumption. It follows that also $\Omega = Q/P$ is independent of the $d_i$, which is what we wanted to show.  
\end{proof}

\begin{lem}\label{lem : polydivisible}
Let $P,Q\in k[x_1,\ldots,x_n]$ be elements of a polynomial ring over the field $k$. Suppose that $P$ is irreducible and vanishes at the origin. Then if $P$ divides $Q$ in the power series ring $k\psl x_1,\ldots,x_n\psr$, it already divides $Q$ in $k[x_1,\ldots,x_n]$.
\end{lem}
\begin{proof}
Let $I=(x_1,\ldots,x_n)$ be the ideal of the origin and consider the localization $k[x_1,\ldots,x_n]_I$. Since the completion $k[x_1,\ldots,x_n]_I\to k\psl x_1,\ldots,x_n\psr$ is faithfully flat, it follows that $P$ divides $Q$ in the power series ring if and only if it does so in the localization.\footnote{Note that the divisibility can be restated as the inclusion of ideals $(P)\subseteq (P,Q)$ being an isomorphism. By flatness the ideals play well with base change. By faithful flatness isomorphisms are preserved and reflected under base change.} In particular, there is some equality of the form $P\,R_1 = Q\, R_2$ in $k[x_1,\ldots,x_n]$, where $R_2\not\in I$. Since $P\in I$ by assumption, it cannot divide $R_2$, therefore $P$ divides $Q$ by unique factorization.   
\end{proof}

Using Proposition \ref{prop: Gamma's are constant}, one can already deduce Theorem \ref{thm: main theorem} for the Gieseker moduli spaces by taking appropriate limits in the $\varepsilon_i$-variables. We want to go further and prove the desired formula in Situation \ref{sit: framed main}. This will use the following result, which is nothing but a restatement of the rank $1$ case of Theorem \ref{thm: main theorem} for the framed moduli spaces.

\begin{prop}\label{prop: rank1blowup}
Let $\mathsf{W}(t_1,t_2,y,q)$ be the rank one (i.e. $r=1$) generating series for equivariant $\chi_y$-genera on the moduli spaces of framed sheaves on $\PP^2$. More precisely,
\[\mathsf{W}(t_1,t_2,y,q):= \sum_{Y}\prod_{s\in Y}\left( \Theta_y(\tT_1^{-l_{Y}(s)}\tT_2^{a_Y(s)+1})\, \Theta_y(\tT_1^{l_{Y}(s)+1}\tT_2^{-a_Y(s)})\right) q^{\abs{Y}},\]
where the sum ranges over all Young-diagrams $Y$.
Then 
\[\frac{W(t_1,t_2/t_1,y,q)\, W(t_1/t_2,t_2,y,q)}{W(t_1,t_2,y,q)}= \prod_{n=1}^{\infty}(1-(yq)^n)^{-1}.\]
\end{prop}

\begin{proof}
Using the definition of $\Theta_y$, we find that

\[\mathsf{W}(t_1,t_2,y,q) = \sum_{Y}\left(\prod_{s\in Y}\frac{1-y \,t_1^{l_{Y}(s)}t_2^{-1-a_Y(s)}}{1-t_1^{l_{Y}(s)}t_2^{-1-a_Y(s)}} \frac{1-y\,t_1^{-1-l_{Y}(s)}t_2^{a_Y(s)}}{1-t_1^{-1-l_{Y}(s)}t_2^{a_Y(s)}}\right) q^{\abs{Y}}.\]
We define a function $\widetilde{\mathsf{W}}:=\widetilde{\mathsf{W}}(t,q,u,T)$ by renaming variables as follows:\footnote{Warning: The symbol $q$ plays different roles before and after the change of variables. The usage after the change of variables is in accordance with that from \cite{Ra_Wa}.}
\[
    t_1\mapsto t,\quad t_2\mapsto q^{-1},\quad y\mapsto u, \quad q\mapsto T/u.
\]
Then 
\[\widetilde{\mathsf{W}}(t,q,u,T) = \sum_{Y}\left(\prod_{s\in Y} \frac{(1-u\, t^{l_Y(s)}q^{1+a_Y(s)})\,(1-u^{-1}t^{1+l_Y(s)}q^{a_Y(s)})}{(1-t^{l_Y(s)}q^{1+a_Y(s)})\,(1-t^{1+l_Y(s)}q^{a_Y(s)})}  \right)T^{\abs{Y}}.\]

The desired formula now becomes 
\begin{equation}\label{eq:rank1lemma}\frac{\widetilde{\mathsf{W}}(tq,q,u,T)\,\widetilde{\mathsf{W}}(t,tq,u,T)}{\widetilde{\mathsf{W}}(t,q,u,T)} = \prod_{n=1}^{\infty} (1-T^n)^{-1}.\end{equation}
By the $(q, t)$-Nekrasov–Okounkov formula \cite[Thm. 1.3]{Ra_Wa}, we have the product expansion

\[\widetilde{\mathsf{W}}(t,q,u,T) = \prod_{i,j,k\geq 1}\frac{(1-uq^it^{j-1}T^k)(1-u^{-1}q^{i-1}t^j T^k)}{(1-q^{i-1}t^{j-1}T^k)(1-q^it^j T^k)}.\]
Expanding each occurrence of $\widetilde{\mathsf{W}}$ according to this formula and keeping track of all the factors proves \eqref{eq:rank1lemma}.
\end{proof}

\begin{proof}[{Proof of Theorem \ref{thm: main theorem} (the Main Theorem)}]
Proposition \ref{prop: Gamma's are constant} states that each element of the collection $\{\Omega_n\}$ is a power series only in $c_1(\tT_1)$ and $c_1(\tT_2)$ over $\coeffs$. In particular, each $\Omega_n(\mathcal{E})$ is pulled back from $A^*(\pt)$ and does not involve any of the equivariant parameters $e_i$.
Consequently, the main equation from Theorem \ref{thm: omega_j structure theorem} (the Weak Structure Theorem) implies that we have a relation of generating series
\begin{align}
    \label{eq: main theorem proof gen fun}
    \sum_n \int_{\Mhat(r-k\epsilon_0-n[\pt])}\Theta(\T) \, q^{2r n-k(r+k)} 
    =
    \Big(\sum_n \Omega_n\, q^{2rn}\Big)
    \cdot
    \sum _n\int_{M(r-n[\pt])}\Theta(\T) \, q^{2r n}.
\end{align}
We wish to show that $\sum_n \Omega_n q^n$ is equal to $\mathsf{Y}_k$ as defined in the statement of the Main Theorem. 
To do this, we will evaluate the equivariant integrals on both sides of equation \eqref{eq: main theorem proof gen fun} and specialize the equivariant parameters corresponding to the $\eT_i$.

The integrals arising in equation \eqref{eq: main theorem proof gen fun} are computed by summing over the fixed points of the associated moduli spaces. By Theorems \ref{thm: fixed point and tangent spaces} and \ref{thm: fixed point and tangent spaces blowup} the contribution from each fixed point is a product of terms of the form $\theta(e_{b}/e_{a}\,t_1^{i_1}t_{2}^{i_2})$, where
\[
    \theta(x)
    :=
    \frac{1-y\,x^{-1}}{1-x^{-1}},
\]
 with $e_i=\ch(\eT_i)$ and $t_i=\ch(\tT_i)$.
Now, taking limits in a fixed order, we observe
    \begin{align}\label{eq: e limits in order}
        \lim_{e_r\to 0}\cdots \lim_{e_{1}\to 0}
        \theta\big(e_{b}/e_{a}\,t_1^{i_1}t_{2}^{i_2}\big)~
        = 
        \begin{cases}
        	1, &\mbox{ if } a < b \\
        	y, &\mbox{ if } a > b \\
        	\dfrac{1-y\, t_1^{-i_1}t_2^{-i_2}}{1-t_1^{-i_1}t_2^{-i_2}}=\theta(t_1^{i_1}t_2^{i_2}),& \mbox{ if } a = b. 
        \end{cases}
    \end{align}

We now consider the generating series associated to the moduli spaces $\Mc$. By the definition of the integral notation from Notation \ref{not: integral notation} combined with Theorem \ref{thm: fixed point and tangent spaces} (i) we have:
\[
    \sum_n\int_{M(r-n[\pt])}\Theta(\T) \, q^{ 2rn}
    =
    \sum_{\bm{Y}} \Theta\big(\T|_{\bm{Y}}\big)
    \, q^{2r\sum_{i=1}^r \abs{Y_i}},
\]
where the sum ranges over $r$-tuples of Young-diagrams $\bm{Y} = (Y_1,\ldots,Y_r)$. Now, by Theorem \ref{thm: fixed point and tangent spaces} (ii), we can compute $\Theta\big(\T|_{\bm{Y}}\big)$ to be:
\begin{align*}
     \prod_{a,b = 1}^r
     \left(
        \prod_{s\in Y_{a}} \Theta\Big(\eT_{b}/\eT_{a} \tT_1^{-l_{Y_{b}}(s)}\tT_2^{a_{Y_{a}}(s)+1} \Big) 
        \times
        \prod_{t\in Y_{b}} \Theta\Big(\eT_{b}/\eT_{a} \tT_1^{l_{Y_{a}}(t) +1 }\tT_2^{-a_{Y_{b}}(t)} \Big) 
     \right). 
\end{align*}
Using equation \eqref{eq: e limits in order} gives that 
$\lim_{e_r\to 0}\cdots \lim_{e_{1}\to 0}\Theta\big(\T|_{\bm{Y}}\big)$ 
is equal to:
\begin{gather*}
    \prod_{1\leq b < a \leq r} 
    y^{\abs{Y_{a}}+\abs{Y_{b}}}
    ~\cdot~
    \prod_{a = 1}^r \prod_{s\in Y_{a}}
    \theta\Big(t_1^{-l_{Y_{a}}(s)}t_2^{a_{Y_{a}}(s)+1}\Big) \theta\Big(t_1^{l_{Y_{a}}(s)+1}t_2^{-a_{Y_{a}}(s)}\Big)\\
    = \prod_{a=1}^r \left(y^{(r-1)\abs{Y_a}} \prod_{s\in Y_a} \theta\Big(t_1^{-l_{Y_{a}}(s)}t_2^{a_{Y_{a}}(s)+1}\Big) \theta\Big(t_1^{l_{Y_{a}}(s)+1}t_2^{-a_{Y_{a}}(s)}\Big)\right).
\end{gather*}
We conclude that:
\begin{equation}
\label{eq:limitframedP2}
\begin{aligned}
    &\lim_{e_r\to 0}\cdots \lim_{e_{1}\to 0}~
    \sum_n\int_{M(r-n[\pt])}\Theta(T) \; q^{2rn}\\
    =& \left(
    \sum_{Y} \left( \prod_{s\in Y} \theta\Big(t_1^{-l_{Y}(s)}t_2^{a_{Y}(s)+1}\Big) \theta\Big(t_1^{l_{Y_{a}}(s)+1}t_2^{-a_{Y_{a}}(s)}\Big)\right)(q^{2r}y^{r-1})^{\abs{Y}}\right)^r\\
    =&\centering \left(\mathsf{W}(t_1,t_2,y,y^{r-1}q^{2r})\right)^r,
\end{aligned}
\end{equation}
where $\mathsf{W}$ is as in Proposition \ref{prop: rank1blowup}.

We now consider the generating series associated to the moduli spaces $\Mhatc$. By the definition of the integral notation from Notation \ref{not: integral notation} combined with Theorem \ref{thm: fixed point and tangent spaces blowup} (i) we have:
\begin{equation}    \label{eq: Mhat gen fun comp}
\begin{gathered}
    \sum_n \int_{\Mhat(r-n[\pt]-k\epsilon_0)}\Theta(\T) \; q^{2rn-k(r+k)}\\  
    =~
    \sum_{(\bm{Y},\bm{Z},\bm{k})} \Theta(\T|_{(\bm{Y},\bm{Z},\bm{k})}) q^{2r\sum_{i=1}^r(\abs{Y_i}+\abs{Z_i})+ \sum_{1\leq i<j\leq r} (k_i-k_j)^2}.
\end{gathered}\end{equation}
Consider a fixed triple $(\bm{Y},\bm{Z},\bm{k})$. In the notation of Theorem \ref{thm: fixed point and tangent spaces blowup} (ii), we have
\begin{align*}
    \lim_{e_r\to 0}\cdots \lim_{e_{1}\to 0}
	\Theta(\mathrm{L}_{a,b}(\tT_1,\tT_2)) 
	&= 
    {\begin{cases}
    	y^{\binom{k_{a}-k_{b}+1}{2}}  &, \mbox{ if } a > b \mbox{ and } k_{a}>k_{b}\\
    	y^{\binom{k_{b}-k_{a}}{2}}  &, \mbox{ if } a > b \mbox{ and } k_{a}+1<k_{b}\\
    	 1 &, \mbox{ else.}
    \end{cases}}\\
    &= {\begin{cases}
    y^{(k_{a}-k_{b})^2/2}y^{(k_{a}-k_{b})/2} &, \mbox{ if } a > b;\\
    1 &, \mbox{ else.}
    \end{cases}}
\end{align*}
In particular, we have the identity
\[
    \lim_{e_r\to 0}\cdots \lim_{e_{1}\to 0}
    \prod_{a,b 
    =1}^r\Theta(\mathrm{L}_{a,b}(\tT_1,\tT_2)) 
    ~=~ 
    y^{\sum_{a>b}(k_{a}-k_{b})^2/2}y^{\sum_{a>b}(k_{a}-k_{b})/2 }.
\]

Proceeding as before for $\N^{\bm{Y}}_{a,b}(-,-)$ gives that $\lim_{e_r\to 0}\cdots \lim_{e_{1}\to 0}\Theta(\T|_{(\bm{Y},\bm{Z},\bm{k})})$ is equal to
\begin{align*}
    &\prod_{a = 1}^r\left( 
    y^{\abs{Y_a}(r-1)}
    \prod_{s\in Y_{a}}
    \theta\Big(t_1^{-l_{Y_{a}}(s)}(t_2/t_1)^{a_{Y_{a}}(s)+1}\Big) \theta\Big(t_1^{l_{Y_{a}}(s)+1}(t_2/t_1)^{-a_{Y_{a}}(s)}\Big)\right)
    \\
    &~\cdot~
    \prod_{a=1}^r\left(  
    y^{\abs{Z_a}(r-1)}
    \prod_{s\in Z_{a}}
    \theta\Big((t_1/t_2)^{-l_{Z_{a}}(s)}t_2^{a_{Z_{a}}(s)+1}\Big) \theta\Big((t_1/t_2)^{l_{Z_{a}}(s)+1}t_2^{-a_{Z_{a}}(s)}\Big)\
    \right)
    \\
    &~\cdot~
    y^{\sum_{a>b}(k_{a}-k_{b})^2/2}y^{\sum_{a<b}(k_{a}-k_{b})/2 }.
\end{align*}
Now, considering the generating series on the left hand side of \eqref{eq: Mhat gen fun comp}, we see that we can separate the sums over pairs of $r$-tuples of Young diagrams into $2r$ sums over Young diagrams after taking limits in the $e_i$.  In conclusion, we find that the limiting series is equal to
\begin{align*}
    \lim_{e_r\to 0}\cdots \lim_{e_1\to 0} \sum_n \int_{\Mhat(r-n[\pt]-k\epsilon_0)}&\Theta(\T) \; q^{2rn-k(r+k)}
    \\
    =&~
    \Big( \mathsf{W}(t_1/t_2,t_2,y,y^{r-1}q^{2r})\cdot \mathsf{W}(t_1,t_2/t_1,y,y^{r-1}q^{2r})\Big)^r 
    \\
    &~\cdot~
    \sum_{k_1+\cdots+k_r = k} (q^2y)^{\sum_{ i<j} (k_i-k_j)^2/2 } \, y^{\sum_{ i<j} (k_i-k_j)/2 }.
\end{align*}
Comparing this to equation \eqref{eq:limitframedP2} and using Proposition \ref{prop: rank1blowup} gives the desired result.
\end{proof}

\bibliography{refs}
\bibliographystyle{amsalpha}
\end{document}